\documentclass[12pt]{article}
\usepackage{a4wide}
\usepackage{amsmath,amssymb,amsthm}
\usepackage[mathscr]{eucal}
\usepackage{tikz}
\usetikzlibrary{positioning,automata}
\usetikzlibrary{arrows,calc}
\usetikzlibrary{decorations.markings,plotmarks}
\usepackage{graphics}
\usepackage{color}
\usepackage{hyperref}
\usepackage{dsfont}
\usepackage{enumerate}
\usepackage{mathtools}
\usepackage{cite}

\newtheorem{theorem}{Theorem}
\newtheorem{proposition}[theorem]{Proposition}
\newtheorem{corollary}[theorem]{Corollary}
\newtheorem{lemma}[theorem]{Lemma}
\newtheorem{fact}[theorem]{Fact}

\theoremstyle{definition}
\newtheorem{definition}[theorem]{Definition}

\newtheorem{conjecture}[theorem]{Conjecture}
\newtheorem{remark}[theorem]{Remark}

\numberwithin{theorem}{section}

\DeclareMathOperator{\Aut}{Aut}

\DeclareMathOperator{\Sym}{Sym}
\DeclareMathOperator{\dist}{dist}

\DeclareMathOperator{\mindeg}{mindeg}
\DeclareMathOperator{\motion}{motion}
\DeclareMathOperator{\spec}{spec}

\newcommand{\footremember}[2]{%
    \footnote{#2}
    \newcounter{#1}
    \setcounter{#1}{\value{footnote}}%
}

\title{A characterization of Johnson and Hamming graphs and proof of Babai's conjecture}
\author{Bohdan Kivva\footremember{alley}{University of Chicago, e-mail: bkivva@uchicago.edu}}

\begin{document}
\maketitle

\vspace{-0.8cm}
\begin{abstract}
One of the central results in the representation theory of distance-regular graphs classifies distance-regular graphs with $\mu\geq 2$ and second largest eigenvalue $\theta_1= b_1-1$. In this paper we give a classification under the (weaker) approximate eigenvalue constraint $\theta_1\geq (1-\varepsilon)b_1$ for the class of geometric distance-regular graphs.  As an application, we confirm Babai's conjecture on the minimal degree of the automorphism group of distance-regular graphs.
  
\end{abstract}


\section{Introduction}
 
In this paper we characterize Johnson and Hamming graphs as geometric distance-regular graphs satisfying certain relaxed spectral constraints. Classical characterization results of Hamming graphs $H(d, q)$ assume equality constraints on certain parameters such as the assumption $\theta_1 = b_1-1$ on the second largest eigenvalue (Theorem~\ref{thm:intro-classif} below) or the assumption $n = (\lambda+2)^d$ on the number of vertices (Enomoto \cite{Enomoto} and Egawa \cite{Egawa}). The principal novelty of our result is that we make no such tight assumptions. We apply our characterization to confirm Babai's conjecture on the minimal degree of the automorphism group of distance-regular graphs.

\subsection{Characterization of Johnson and Hamming graphs}

 A result of Terwilliger~\cite{ter-local} (see {\cite[Theorem 4.4.3]{BCN}}) implies that the icosahedron is the only distance-regular graph, for which the second largest eigenvalue $\theta_1$ (of the adjacency matrix) satisfies $\theta_1>b_1-1$ and a pair of vertices at distance 2 has $\mu\geq 2$ common neighbors. Another classical result gives the classification of distance-regular graphs with $\mu\geq 2$ and $\theta_1 = b_1-1$.      

\begin{theorem}[{\cite[Theorem 4.4.11]{BCN}}]\label{thm:intro-classif}
Let $X$ be a distance-regular graph of diameter $d\geq 3$ with second largest eigenvalue $\theta_1 = b_1-1$. Assume $\mu\geq 2$. Then one of the following holds:
\begin{enumerate}
\item $\mu =2$ and $X$ is a Hamming graph, a Doob graph, or a locally Petersen graph (and all such graphs are known).
\item $\mu = 4$ and $X$ is a Johnson graph.
\item $\mu = 6$ and $X$ is a half cube.
\item $\mu =10$ and $X$ is a Gosset graph $E_7(1)$.
\end{enumerate}
\end{theorem}

In this paper we consider the case $\theta_1\geq (1-\varepsilon)b_1$ for sufficiently small $\varepsilon>0$. The relaxation of the assumption on the second largest eigenvalue comes at the cost of requiring additional structural constraints.
Our main structural assumption is that $X$ is a geometric distance-regular graph, meaning that there exists a collection of Delsarte cliques (see Sec.~\ref{sec:geometric}) $\mathcal{C}$ such that every edge of $X$ belongs to a unique clique in $\mathcal{C}$. Additional technical structural assumptions depend on whether the neighborhood graphs of $X$ are connected. We note that for a geometric distance-regular graph $X$ either the neighborhood graph $X(v)$ is connected for every vertex $v$, or $X(v)$ is disconnected for every vertex $v$ (see Lemma~\ref{lem:connected-disconn}). We give the following two characterizations.

\begin{theorem}\label{thm:pseudo-main-intr}
 Let $X$ be a geometric distance-regular graph of diameter $d\geq 2$ with smallest eigenvalue $-m$. Suppose that $\mu\geq 2$ and $\theta_1+1> (1-\varepsilon^*)b_1$ for an absolute constant $\displaystyle{0<\varepsilon^*\approx 0.0065}$. Moreover, assume that the vertex degree satisfies $k\geq \max(m^3, 29)$ and the neighborhood graph $X(v)$ is connected for some vertex $v$ of $X$. 
 
 Then $X$ is a Johnson graph $J(s, d)$ for $s=(k/d)+d\geq 2d+1$.
\end{theorem}

\begin{remark} We give the exact definition of $\varepsilon^*$ in Proposition \ref{prop:local-structure} (see also Def.~\ref{def:eig_bounds} and Theorem~\ref{thm:theta1}). We note that $\varepsilon^*$ can be set to be as large as $2/7$, if we additionally assume $k$ to be sufficiently large (see Remark~\ref{rem:varepsilon27}).
\end{remark}
   
\begin{theorem}\label{thm:main-Hamming-intr} Let $X$ be a geometric distance-regular graph of diameter $d\geq 2$ with smallest eigenvalue $-m$. Suppose that $\mu\geq 2$ and  $\theta_{1}\geq (1-\varepsilon)b_1$, where $0<\varepsilon< 1/(6m^4d)$. Moreover, assume  $c_t\leq \varepsilon k$ and $b_t\leq \varepsilon k$ for some $t\leq d$, and the neighborhood graph $X(v)$ is disconnected for some vertex $v$ of $X$.
 
 Then $X$ is a Hamming graph $H(d, s)$, for $s = 1+k/d$.
\end{theorem}

\begin{remark} If $s>6d^5+1$, then the Hamming graph $H(d, s)$ satisfies the assumptions of this theorem with $1/(s-1)\leq \varepsilon<1/(6d^5)$ and $t = d$.
\end{remark}

The assumption that a distance-regular graph is geometric excludes only finitely many graphs with $\mu\geq 2$, if the smallest eigenvalue of the graph is assumed to be bounded, as proved by  Bang and Koolen~\cite{Bang-Koolen-conj} in 2010.

\begin{theorem}[Koolen, Bang]
Fix an integer $m\geq 2$. There are only finitely many non-geometric distance-regular graphs of diameter $\geq 3$ with $\mu\geq 2$ and smallest eigenvalue at least $-m$. 
\end{theorem}

Moreover, they conjecture the following classification of geometric distance-regular graphs with fixed smallest eigenvalue $-m$ and $\mu\geq 2$. 

\begin{conjecture}[Koolen, Bang]\label{conj:Bang-Koolen}
For a fixed integer $m\geq 2$, any geometric distance-regular graph with smallest eigenvalue $-m$, diameter $\geq 3$ and $\mu\geq 2$ is a Johnson graph, or a Hamming graph, or a Grassmann graph, or a bilinear forms graph, or the number of vertices is bounded by a function of $m$.
\end{conjecture}

Our characterizations confirm this conjecture in rather special cases.

Theorem~\ref{thm:main-Hamming-intr} is one of the main technical contributions of the paper. Even though the assumptions of Theorem~\ref{thm:pseudo-main-intr} seem weaker (for instance, $\varepsilon$ is absolute), we believe that, in comparison with known results, Theorem~\ref{thm:main-Hamming-intr} brings more novelty. The known characterization of Johnson graphs in terms of the local structure (Theorem~\ref{thm:Johnson-classif}) seems to be more easily applicable than the known characterizations of Hamming graphs. All characterizations of Hamming graphs known to the author, in terms of intersection numbers, eigenvalues or local structure, make strong equality constraints either on the number of vertices, or on the eigenvalues. In contrast,  Theorem~\ref{thm:main-Hamming-intr} makes no assumptions of such flavor and therefore might be more broadly applicable.

\subsection{Motion: Minimal degree of the automorphism group}

For a permutation $\sigma$ of a set $\Omega$ the number of points not fixed by $\sigma$ is called the \textit{degree} of the permutation $\sigma$. Let $G$ be a permutation group on a set $\Omega$. The minimum of the degrees of non-identity elements of $G$ is called the \textit{minimal degree} of $G$. We denote this quantity by $\mindeg(G)$. Lower bounds on the minimal degree of a group have strong constraints on the structure of the group. A result of  Wielandt \cite{Wielandt} implies that a linear lower bound on $\mindeg(G)$ provides a logarithmic upper bound on the degree of an alternating group involved in $G$ as a quotient of a subgroup. 
The study of the minimal degree goes back to 19th century (Jordan~\cite{Jordan}, Bochert~\cite{Bochert}). We briefly indicate some of the history and the motivation of this classical concept in the Outlook section (Sec.~\ref{sec:outlook}).

\begin{definition}[Motion]\label{def-motion}
Following~\cite{RS-motion}, for a combinatorial structure $\mathcal{X}$ we use term \textit{motion} for the minimal degree of the automorphism group $\Aut(\mathcal{X})$,
\[\motion(\mathcal{X}) := \mindeg(\Aut(\mathcal{X})).\]
\end{definition}

In 2014 Babai \cite{babai-srg}, \cite{babai-srg-2} classified the strongly regular graphs with sublinear motion.

\begin{theorem}[Babai]\label{thm:babai-srg} 
Let $X$ be a strongly regular graph on $n$ vertices. Then either
\[ \motion(X)\geq \frac{n}{8}\,,\]
or $X$, or its complement is the Johnson graph $J(s, 2)$, the Hamming graph $H(2, s)$ or a disjoint union of cliques of the same size.
\end{theorem}

Non-trivial strongly regular graphs are primitive distance-regular graphs of diameter 2. Babai conjectured that a similar classification should be true for primitive distance-regular graphs of fixed diameter. 

\begin{conjecture}[Babai]\label{conj-dist-reg}

For any $d\geq 3$ there exists $\gamma_d>0$, such that for any primitive distance-regular graph $X$ of diameter $d$ on $n$ vertices either
\[ \motion(X) \geq \gamma_d n,\]
or $X$ is the Johnson graph or the Hamming graph.
\end{conjecture}

Using our characterizations of Johnson and Hamming graphs, we confirm the conjecture in the stronger form dropping the primitivity assumption. (The imprimitive case admits one more class of exceptions, the cocktail-party graphs). A large portion of the proof relies on the results we prove in~\cite{kivva-drg}.

\begin{theorem}\label{thm:main-motion} For any $d\geq 3$ there exists $\gamma_d>0$, such that for any distance-regular graph $X$ of diameter $d$ on $n$ vertices either
$$\motion(X)\geq \gamma_d n,$$
or $X$ is a Johnson graph, or a Hamming graph, or a cocktail-party graph. 
\end{theorem}

 In~\cite{kivva-drg} we reduced the case of imprimitive distance-regular graphs to confirming the conjecture for primitive graphs. Moreover, in one of the main results of~\cite{kivva-drg} we show that it is sufficient to study geometric distance-regular graphs with bounded smallest eigenvalue, which is the case we settle in this paper. 

\begin{theorem}[{\cite[Theorem 1.6]{kivva-drg}}]\label{thm:main-motion-cited}
For any $d\geq 3$ there exist $\gamma_d'>0$ and a positive integer $m_d$, such that for any primitive distance-regular graph $X$ of diameter $d$ on $n$ vertices  either 

 \[\motion(X)\geq \gamma_d' n,\]
 or $X$ is geometric with smallest eigenvalue at least $-m_d$. 
\end{theorem}

\begin{remark} The reason we split the proof of Theorem~\ref{thm:main-motion} between two papers is the different nature of techniques and results obtained. In particular, in~\cite{kivva-drg} we prove a spectral gap bound for distance-regular graphs which have a dominant distance. We use it as one of the ingredients in the analysis of $\mu=1$ case in this paper (see Theorem~\ref{thm:cited-main-spectral-gap}). 
\end{remark}



\subsection{A few words on the proof of the characterization theorems}

The proof of Theorem~\ref{thm:intro-classif} has two crucial ingredients. The first ingredient is Terwilliger's $1+b_1/(\theta_1+1)$ lower bound on the smallest eigenvalue of the neighborhood graphs. In the case when $\theta_1 = b_1-1$ this bound implies that the smallest eigenvalue of each neighborhood graph is at least $-2$ and all such graphs are classified by Cameron, Goethals, Seidel and Shult~\cite{regular-classif} (note that the neighborhood graphs are regular). In particular, if a neighborhood graph has no induced quadrangle, then the neighborhood graph is a clique, the pentagon or the Peterson graph.  The second ingredient is the existence of a root representation (an integral affine (2, 4)-representation) for a distance-regular graph with $\theta_1 = b_1-1$ and which contains a quadrangle. The classification of the amply regular graphs with $\mu\geq 2$, which have a root representation, is known (see {\cite[Theorems 3.15.1-3.15.4]{BCN}}).

We note that in the case of geometric distance-regular graphs the condition $\mu\geq 2$ implies the existence of an induced quadrangle.  The argument that shows the existence of the root representation heavily depends on the equality constraint $\theta_1 = b_1-1$ and does not survive when this condition is relaxed to $\theta_1\geq (1-\varepsilon)b_1$. However, under our assumptions, we are still able to make use of the first ingredient. Indeed, results of Hoffman \cite{Hoffman} and Bussemaker, Neumaier \cite{Bussem-Neum} imply that no graph has the smallest eigenvalue in the range $(-2-\delta, -2)$ for some constant $\delta>0$. This will be sufficient to prove Theorem~\ref{thm:pseudo-main-intr} (see Section~\ref{sec:Johnson}).

The proof of Theorem~\ref{thm:main-Hamming-intr} requires a different approach. A distance-regular graph is geometric with disconnected neighborhood graphs if and only if each of its neighborhood graphs is a disjoint union of cliques. We prove that if $X$ satisfies the assumptions of Theorem~\ref{thm:main-Hamming-intr} and is not a Hamming graph, then the second largest eigenvalue has multiplicity less than the vertex degree of $X$. Terwilliger \cite{ter-local} proved that this implies that each neighborhood graph has an eigenvalue less than $-1$. This leads to a contradiction, as the least eigenvalue of a disjoint union of cliques is $-1$. We prove Theorem~\ref{thm:main-Hamming-intr} in Section~\ref{sec:Hamming}.

\section*{Acknowledgments}

The author is grateful to Professor L\'aszl\'o Babai for introducing him to the problems discussed in the paper, for his invaluable assistance in framing the results, which significantly improved the presentation of the paper, and for his constant support and encouragement.

\section{Preliminaries}

In this section we give definitions and basic properties for the objects and notions that are used in the paper. For more about distance-regular graphs we refer the reader to the monograph \cite{BCN} and the survey  article \cite{Koolen-survey}. 

\subsection{Basic concepts and notation for graphs}

Let $X$ be a graph. We always denote by $n$ the number of vertices of $X$ and for a regular graph $X$ we denote by $k$ its degree.  We denote the diameter of $X$ by $d$. If the graph is disconnected, then its diameter is defined to be $\infty$. 

\begin{definition} A regular graph is called \textit{edge-regular} if every pair of adjacent vertices has the same number $\lambda$ of common neighbors.
\end{definition}
\begin{definition} A regular graph is called \textit{amply regular} if every pair of vertices at distance 2 has the same number $\mu$ of common neighbors.
\end{definition}

Let $N(v)$ be the set of neighbors of a vertex $v$ in $X$ and $N_i(v) = \{ w\in X| \dist(v,w) = i\}$ be the set of vertices at distance $i$ from $v$ in the graph $X$. By $X(v)$ we denote the \textit{neighborhood graph} of $v$ in $X$, namely, the graph induced by $X$ on $N(v)$. 

\begin{definition} Let $X$ be a graph. Denote by $V(X)$ and $E(X)$ the set of vertices and the set of edges of $X$. The \textit{line graph} of $X$ is the graph $L(X)$ with the set of vertices $E(X)$, in which distinct $e_1, e_2\in E(X)$ are adjacent if they (as edges of $X$) share a vertex. 
\end{definition}

By \textit{eigenvalues of a graph} we mean the eigenvalues of its \textit{adjacency matrix}.

\subsection{Distance-regular graphs}\label{sec:drg}

\begin{definition}
 A connected graph $X$ of diameter $d$ is called \textit{distance-regular} if for every $0\leq i\leq d$ there exist constants $a_i, b_i, c_i$ such that for any $v\in X$ and any $w\in N_i(v)$ the number of edges between $w$ and $N_i(v)$ is $a_i$, between $w$ and $N_{i-1}(v)$  is $c_i$, and between $w$ and $N_{i+1}(v)$ is $b_i$. The sequence
 \[\iota(X) = \{b_0, b_1,\ldots, b_{d-1}; c_1, c_2,\ldots, c_d\}\]
 is called the \textit{intersection array} of $X$.
\end{definition}

Note that every distance-regular graph is edge-regular and amply regular with $\lambda = a_1$ and $\mu = c_2$.


 By simple counting, the following properties of the parameters of distance-regular graphs hold.
\begin{enumerate}
\item $a_i+b_i+c_i = k$ for every $0\leq i\leq d$,
\item $b_{i+1} \leq b_{i}$ and $c_{i+1}\geq c_{i}$ for $0\leq i\leq d-1$.
\item $|N_{i}(v)|b_i = |N_{i+1}(v)|c_{i+1}$, for $0\leq i\leq d-1$ and any vertex $v$. 
\end{enumerate}
Thus, the numbers $k_i = |N_i(v)|$ do not depend on a vertex $v\in X$, and we can rewrite the last property as
\begin{enumerate}
\item[3'.] $k_i b_i = k_{i+1}c_{i+1}$, for $0\leq i\leq d-1$. 
\end{enumerate}

From the definition of a distance-regular graph it can be deduced that there are parameters $p_{i,j}^s$ such that for any $u,v\in X$ with $\dist(u,v) = s$ there exist precisely $p_{i,j}^{s}$ vertices at distance $i$ from $u$ and distance $j$ from $v$, i.e., $|N_i(u)\cap N_j(v)| = p_{i,j}^{s}$. The parameters $p_{i,j}^s$ are called \textit{intersection numbers}.

\begin{lemma}\label{lem:lambda-mu}
Let $X$ be a distance-regular graph of diameter $d\geq 2$. Then $2\lambda\leq k+\mu$.
\end{lemma}
\begin{proof}
Denote $N(x, y) =N(x)\cap N(y)$ for vertices $x$ and $y$ of $X$. The inequality above follows from the two obvious inclusions below applied for $v$ and $w$ at distance $2$ and their common neighbor $v$
\[ N(u, v) \cup N(u, w)\subseteq N(u),\qquad\quad  N(u, v) \cap N(u, w) \subseteq N(v, w).\] 
\end{proof}

In Section \ref{sec:Hamming} we will need the following inequality proved by Terwilliger~\cite{terwilliger}.

\begin{theorem}[Terwilliger \cite{terwilliger}, see {\cite[Theorem 5.2.1]{BCN}}]\label{thm:ter-inequality} Let $X$ be a distance-regular graph. If $X$ contains an induced quadrangle, then 
\[ c_{i}-b_{i}\geq c_{i-1}-b_{i-1}+\lambda+2, \quad \text{for } i= 1, 2, \ldots, d.\]
\end{theorem}
We also need the following result regarding the increase $c_i-c_{i-1}$.
\begin{theorem}[{\cite[Theorem 5.4.1]{BCN}}]
Let $X$ be a distance-regular graph of diameter $d\geq 3$. If $\mu\geq 2$, then either $c_3\geq 3\mu/2$, or $c_3\geq \mu+b_2$ and $d=3$.
\end{theorem}
\begin{corollary}\label{cor:c3-mu}
Let $X$ be a distance-regular graph of diameter $d\geq 3$. If $\mu\geq 2$, then $c_3>\mu$.
\end{corollary}
A distance-regular graph $X$ of diameter $d$ has precisely $d+1$ distinct eigenvalues. Denote these eigenvalues by $\theta_0 = k> \theta_1> \ldots > \theta_d$. They are the eigenvalues of the tridiagonal \textit{intersection matrix} below.
 \[L_1 = \left(\begin{matrix} 
a_0 & b_0 & 0 & 0 & ... \\ 
c_1 & a_1 & b_1 & 0 &...\\
0& c_2 & a_2 & b_2 & ...\\
...& & \vdots & & ...\\
...& & 0& c_d & a_d 
\end{matrix}\right).\]
For an eigenvalue $\theta$, consider the sequence $(u_{i}(\theta))_{i=0}^{d}$ defined by the relations
\[ u_{0}(\theta) = 1, \qquad u_{1}(\theta) = \frac{\theta}{k},\]
\[ c_{i}u_{i-1}(\theta)+a_iu_i(\theta)+b_iu_{i+1}(\theta) = \theta u_{i}(\theta), \quad \text{for } i=1, 2, ..., d-1, \]
\[ c_du_{d-1}(\theta)+a_du_{d}(\theta) = \theta u_{d}(\theta).\]
The vector $u = (u_{0}(\theta), u_{1}(\theta), ..., u_{d}(\theta))^T$ is an eigenvector of $L_1$ corresponding to~$\theta$.

\begin{definition}
The sequence $(u_{i}(\theta))_{i=0}^{d}$ is called the \textit{standard sequence} of $X$ corresponding to the eigenvalue $\theta$.
\end{definition}

We denote by $f_i$ the \textit{multiplicity} of the eigenvalue $\theta_i$ of $X$. Since $X$ is a connected graph, $f_0 = 1$. In general, the multiplicities $f_i$ can be computed using the Biggs formula.

\begin{theorem}[Biggs \cite{biggs}, see {\cite[Theorem 4.1.4]{BCN}}] The multiplicity of an eigenvalue $\theta$ of a distance-regular graph $X$ can be expressed as
\[ f(\theta) = \frac{n}{\sum\limits_{i=0}^{d}k_iu_i(\theta)^2}.\]

\end{theorem}

\subsection{Geometric distance-regular graphs}\label{sec:geometric}

Let $X$ be a distance-regular graph, and $\theta_{\min}$ be its smallest eigenvalue. Delsarte proved in~\cite{Delsarte} that any clique $C$ in $X$ satisfies $\displaystyle{|C|\leq 1-\frac{k}{\theta_{\min}}}$. A clique in $X$ of size $\displaystyle{1-\frac{k}{\theta_{\min}}}$ is called a \textit{Delsarte clique}.
\begin{definition}\label{def-geom}
A distance-regular graph $X$ is called \textit{geometric} if there exists a collection $\mathcal{C}$ of Delsarte cliques such that every edge of $X$ lies in exactly one clique of $\mathcal{C}$.
\end{definition}

More generally, we say that a graph contains a \textit{clique geometry}, if there exists a collection $\mathcal{C}_0$ of maximal cliques, such that every edge is contained in exactly one clique of $\mathcal{C}_0$. Metsch proved that a graph satisfying rather modest assumptions contains a clique geometry.

\begin{theorem}[Metsch {\cite[Result 2.2]{Metsch}}]\label{Metsch}
Let $\mu\geq 1$, $\lambda^{(1)}, \lambda^{(2)}$ and $m$ be integers. Assume that $X$ is a connected graph with the following properties.
\begin{enumerate}
\item Each pair of 
adjacent vertices has at least $\lambda^{(1)}$ and at most $\lambda^{(2)}$ common neighbors.
\item Each pair of non-adjacent vertices has at most $\mu$ common neighbors.
\item \quad $2\lambda^{(1)}-\lambda^{(2)}>(2m-1)(\mu-1)-1$.
\item Every vertex has fewer than $(m+1)(\lambda^{(1)}+1)-\frac{1}{2}m(m+1)(\mu-1)$ neighbors.
\end{enumerate}

Define a \emph{line} to be a maximal clique $C$ satisfying $|C|\geq \lambda^{(1)}+2-(m-1)(\mu-1)$. Then every vertex is on at most $m$ lines, and every pair of adjacent vertices lies in a unique line. 

\end{theorem}

The following sufficient condition for being geometric is a slightly reformulated version of a result from \cite{Koolen-survey}.
\begin{proposition}[{\cite[Proposition 9.8]{Koolen-survey}}]\label{suff-cond}
Let $X$ be a distance-regular graph of diameter $d\geq 2$. Assume there exist a positive integer $m$ and a clique geometry $\mathcal{C}$ of $X$ such that  every vertex $u$ is contained in exactly $m$ cliques of $\mathcal{C}$. If $k\geq m^2$, then $X$ is geometric with smallest eigenvalue $-m$.   
\end{proposition} 

The converse holds without the $k\geq m^2$ assumption. 

\begin{lemma}\label{lem:eigenvalue-clique-geom} Let $X$ be a geometric distance-regular graph of diameter $d\geq 2$ with smallest eigenvalue $-m$. Let $\mathcal{C}$ be a Delsarte clique geometry. Then $m$ is an integer and every vertex belongs to precisely $m$ Delsarte cliques in $\mathcal{C}$.  
\end{lemma}
\begin{proof} By the definition of a Delsarte clique, its size is $1+k/m$. Let $C_1, C_2, \ldots, C_t$ be the cliques in $\mathcal{C}$ which contain a vertex $v$. Since $\mathcal{C}$  is a clique geometry, any distinct $C_i$ and $C_j$ for $i, j\in [t]$ have only $v$ in their intersection, and every vertex adjacent to $v$ belongs to one of the cliques $C_1, C_2, \ldots, C_t$. Therefore, $k = t(|C_i|-1) = tk/m$.   
\end{proof}

In the case, when the smallest eigenvalue of a geometric distance-regular graph is $-2$, we can deduce that the graph is a line graph.

\begin{lemma}\label{m2-line}
Let $X$ be a geometric distance-regular graph with smallest eigenvalue $-2$. Then $X$ is the line graph $L(Y)$ for some graph $Y$.
\end{lemma}
\begin{proof}
Let $\mathcal{C}$ be a Delsarte clique geometry of $X$. 
Define the graph $Y$ with the set of vertices $V(Y) = \mathcal{C}$, in which a pair of distinct vertices $C_1, C_2 \in \mathcal{C}$ in $Y$ is adjacent if and only if $C_1\cap C_2 \neq \emptyset$. We claim that $L(Y) \cong X$. Indeed, since every edge of $X$ is in exactly one clique from $\mathcal{C}$, $|C_1\cap C_2| \leq 1$ for any distinct $C_1, C_2 \in \mathcal{C}$. So there is a well-defined map $f:E(Y)\rightarrow V(X)$. Moreover, by Lemma~\ref{lem:eigenvalue-clique-geom}, every vertex of $X$ is in exactly two cliques from $\mathcal{C}$, so $f$ is bijective. Additionally, a pair of edges in $Y$ share a vertex if and only if there is an edge between the corresponding vertices in $X$. Hence, $L(Y) \cong X$. 
\end{proof}

 Suppose that $X$ is a geometric distance-regular graph with a Delsarte clique geometry $\mathcal{C}$. Consider a Delsarte clique $C\in \mathcal{C}$. Assume $x\in X$ satisfies $\dist(x, C) = i$. Define 
 \begin{equation}
 \psi_i(C, x) := |\{y\in C \mid d(x, y) = i\}|.
 \end{equation}
By~\cite{delsarte-geom}, numbers $\psi_i(C, x)$ do not depend on $C$ and $x$, but only on the distance $\dist(x, C) = i$. Thus, we may define $\psi_i := \psi(C, x)$. 
 
 For $x, y\in X$ with $\dist(x, y) = i$ define 
 \begin{equation}
 \tau_i(x, y; \mathcal{C}) = |\{C\in \mathcal{C}\mid x\in C,\, d(y, C) = i-1\}|.
 \end{equation}
  Again, in \cite{delsarte-geom} it is shown that for a geometric distance-regular graph $X$ the number $\tau_i(x, y; \mathcal{C})$ does not depend on the pair $x,y$, but only on the distance $\dist(x, y) = i$. Therefore, we may define $\tau_i := \tau_i(x, y; \mathcal{C})$.
 
\begin{lemma}[Bang, Hiraki, Koolen {\cite[Proposition 4.2]{delsarte-geom}}]\label{lem:geometric-param}
Let $X$ be a geometric distance-regular graph of diameter $d$, with smallest eigenvalue $-m$. Then
\begin{enumerate}
\item $c_i = \tau_i\psi_{i-1}$, for $1\leq i\leq d$;
\item $\displaystyle{b_i = (m-\tau_i)\left(\frac{k}{m}+1-\psi_i\right)}$, for $1\leq i\leq d-1$.
\end{enumerate}
\end{lemma}

\begin{lemma}\label{lem:tau2} Let $X$ be a geometric distance-regular graph of diameter $d\geq 2$. Then
$\tau_2\geq \psi_1$.
\end{lemma}
\begin{proof}
Let $C\in \mathcal{C}$ be a Delsarte clique of $X$ and let $v$ be a vertex with $\dist(v, C) = 1$. Since $C$ is a maximal clique, there exists a vertex $y\in C$ non-adjacent to $v$. Then $\dist(v, y) = 2$. Let $u_1, u_2, \ldots, u_{\psi_1}$ be the neighbors of $v$ in $C$. Denote by $C_i\in \mathcal{C}$ a Delsarte clique that contains $v$ and $u_i$ for $i\in [\psi_1]$. Note that since $C$ intersects each of $C_i$ in at most one vertex, all $C_i$ are distinct. Moreover, $\dist(y, C_i) = 1$, while $\dist(v, y) = 2$. Therefore, $\tau_2\geq \psi_1$.
\end{proof}

\begin{corollary}\label{cor:mu-m2} Let $X$ be a geometric distance-regular graph of diameter $d\geq 2$, with smallest eigenvalue $-m$. Then $\mu\leq m^2$.
\end{corollary}
\begin{proof} $\mu = \tau_2\psi_1\leq \tau_2^2\leq m^2$. 
\end{proof}

\begin{lemma}\label{lem:connected-disconn} Let $X$ be a geometric distance-regular graph of diameter $d\geq 2$. 
\begin{enumerate}
\item If $\psi_1 = 1$, then for each vertex $v\in X$ its neighborhood graph $X(v)$ is a disjoint union of $m$ cliques, where $-m$ is the smallest eigenvalue of $X$.
\item If $\psi_1\geq 2$, then for each vertex $v\in X$ its neighborhood graph $X(v)$ is connected.
\end{enumerate}
Thus, in particular, either each neighborhood graph of $X$ is connected, or each neighborhood graph of $X$ is disconnected.
\end{lemma}
\begin{proof}
Fix a Delsarte clique geometry $\mathcal{C}$ of $X$. Let $C_1, C_2, \ldots, C_m\in \mathcal{C}$ be the cliques which contain $v$. Take $w\in N(v)$ and let $C_i$ be the clique which contains $w$. If $\psi_1 = 1$, then $v$ is the only neighbor of $w$ in $C_j$ for $j\neq i$. Thus $X(v)$ is a disjoint union of $m$ cliques, where by Lemma~\ref{lem:eigenvalue-clique-geom}, $-m$ is the smallest eigenvalue of $X$. If $\psi_1\geq 2$, then $w$ is adjacent with at least one vertex in $C_j$ distinct from $u$ for every $j\neq i$. Thus $X(v)$ is a connected graph in this case. 
\end{proof}

\subsection{Johnson graphs and Hamming graphs}

\subsubsection{Johnson graphs}

\begin{definition}
Let $d\geq 2$ and $\Omega$ be a set of $s\geq 2d$ points. The \textit{Johnson graph} $J(s,d)$ is a graph on the set $V(J(s,d)) = \binom{\Omega}{d}$ of $n = \binom{s}{d}$ vertices, for which two vertices are adjacent if and only if the corresponding subsets $U_1, U_2\subseteq \Omega$ differ by exactly one element, i.e., $|U_1\setminus U_2| = |U_2\setminus U_1| = 1$.
\end{definition}

It is not hard to check that $J(s, d)$ is a distance-regular graph of diameter $d$ with intersection numbers
\begin{equation}
b_i = (d-i)(s-d-i) \quad \text{and} \quad c_{i+1} = (i+1)^{2}, \quad \text{for } 0\leq i<d.
\end{equation} 
In particular, $J(s, d)$ is regular of degree $k = d(s-d)$ with $\lambda = s-2$ and $\mu = 4$.  The eigenvalues of $J(s,d)$ are
\begin{equation}
 \xi_j = (d-j)(s-d-j)-j \quad \text{with multiplicity}\quad \binom{s}{j} - \binom{s}{j-1}, \, \text{for } 0\leq j\leq d.
 \end{equation}
Using Lemma~\ref{lem:geometric-param} it is easy to see that for the Johnson graph $J(s, d)$
\[ \tau_i =i \quad \text{and} \quad \psi_{i-1} =i ,\quad  \text{for } 1\leq i\leq d.\]

For $s\geq 2d+1$, the automorphism group of $J(s,d)$ is the induced symmetric group $S_s^{(d)}$, which acts on $\binom{\Omega}{d}$ via the induced action of $S_s$ on $\Omega$. Indeed, it is clear, that $S^{(d)}_s\leq \Aut(J(s,d))$. The opposite inclusion can be derived from the Erd\H{o}s-Ko-Rado theorem.

Thus, for a fixed $d$ and $s\geq 2d+1$, the order is $|\Aut(J(s,d)| = s! = \Omega(\exp(n^{1/d}))$, the thickness satisfies $\theta(\Aut(J(s, d))) = s = \Omega(n^{1/d})$ and  
\begin{equation}
\motion(J(s,t)) = O(n^{1-1/d}).
\end{equation}

\begin{theorem}[{\cite[Theorem 9.1.3]{BCN}}]\label{thm:Johnson-classif}
Let $X$ be a connected graph such that 
\begin{enumerate}
\item for each vertex $v$ of $X$ the graph $X(v)$ is the line graph of  $K_{s, t}$; 
\item if $\dist(x,y) = 2$, then $x$ and $y$ have at most $4$ common neighbors. 
\end{enumerate}
Then $X$ is a Johnson graph or is doubly covered by a Johnson graph. More precisely, in the latter case $X$ is the quotient of the Johnson graph $J(2d, d)$ by an automorphism of the form $\tau\omega$, where $\tau$ is the automorphism sending each $d$-set to its complement, and $\omega$ is an element of order at most $2$ in $Sym(X)$ with at least $8$ fixed points.
\end{theorem}

\subsubsection{Hamming graphs}


\begin{definition}
Let $\Omega$ be a set of $s\geq 2$ points. The \textit{Hamming graph} $H(d, s)$ is a graph on the set $V(H(d, s)) = \Omega^{d}$ of $n = s^d$ vertices, for which a pair of vertices is adjacent if and only if the corresponding $d$-tuples $v_1, v_2$ differ in precisely one position. In other words, if the Hamming distance $d_H(v_1, v_2)$ for the corresponding tuples equals 1.
\end{definition}
Again, it is not hard to check that $H(d,s)$ is a distance-regular graph of diameter $d$ with intersection numbers
\begin{equation}\label{eq:hamming}
 b_i =(d-i)(s-1) \quad \text{and} \quad c_{i+1} = i+1 \quad \text{for } 0\leq i\leq d-1.
 \end{equation}
In particular, $H(d,s)$ is regular of degree $k = d(s-1)$ with $\lambda = s-2$ and $\mu = 2$. The eigenvalues of $H(d, s)$ are
\begin{equation}
 \xi_j = d(s-1) - js \quad \text{with multiplicity}\quad \binom{d}{j}(s-1)^{j}, \, \text{for } 0\leq j\leq d.
 \end{equation}
Using Lemma~\ref{lem:geometric-param} it is easy to see that for the Hamming graph $H(d, s)$
\[ \tau_i =i \quad \text{and} \quad \psi_{i-1} =1 ,\quad  \text{for } 1\leq i\leq d.\]

The automorphism group of $H(d, s)$ is isomorphic to the wreath product $S_s\wr S_d$. Hence,  its order is $|\Aut(H(d, s))| = (s!)^d d!$, the thickness satisfies $\theta(H(d,s)) \geq s = n^{1/d}$ and 
\begin{equation}
\motion(H(d, s))\leq 2s^{d-1} = O(n^{1-1/d}).
\end{equation}

In Section~\ref{sec:Hamming} we use the classification of distance-regular graphs with the same intersection array as Hamming graphs. In the case of diameter $2$, the unique non-Hamming graph that has the intersection array of a Hamming graph is the Shrikande graph. It has 16 vertices and has the same parameters as $H(2, 4)$.  

\begin{definition} The direct product of a Hamming graph $H(t, 4)$ with $\ell\geq 1$ copies of the Shrikande graph is called a \textit{Doob graph}. 
\end{definition}
One can check that Doob graphs are distance-regular and have the same intersection numbers as the Hamming graph $H(t+2\ell, 4)$. Yoshimi Egawa \cite{Egawa} proved that Doob graphs are the only graphs with this property. 

\begin{theorem}[Egawa~\cite{Egawa}, see {\cite[Corollary 9.2.5]{BCN}}]\label{thm:Hamming-array-class} A distance-regular graph of diameter $d$ with intersection numbers given by Eq.~\eqref{eq:hamming} is a Hamming graph or a Doob graph.
\end{theorem}

\subsection{Dual graphs}

Let $X$ be a distance-regular graph which has a clique geometry $\mathcal{C}$.  

\begin{definition} By a \textit{dual graph of $X$} \textit{(that corresponds to $\mathcal{C}$)} we mean the graph $\widetilde{X}$ with the vertex set $\mathcal{C}$, in which $C_i$ and $C_j$ are adjacent if and only if $|C_i\cap C_j| = 1$. 
\end{definition}

\begin{lemma}\label{lem:dual-grapp-degree} Let $X$ be a geometric distance-regular graph of diameter $d$ with smallest eigenvalue $-m$. Then its dual graph $\widetilde{X}$ is an edge-regular graph of diameter $d-1$ with the vertex degree $\displaystyle{\widetilde{k} = (m-1)\left(\frac{k}{m}+1\right)}$ and $\displaystyle{\widetilde{\lambda} = (m-2)+(\psi_1-1)\frac{k}{m}}$.
\end{lemma}
\begin{proof} It is known that every Delsarte clique is completely regular, with covering radius $d-1$, see {\cite[Lemma 7.2]{Godsil}}. In particular, this implies that the diameter of $\widetilde{X}$ is $d-1$. Every clique in the Delsarte clique geometry $\mathcal{C}$ of $X$ has size $1+k/m$ and, by Lemma~\ref{lem:eigenvalue-clique-geom}, every vertex is in precisely $m$ cliques from $\mathcal{C}$. Since every pair of non-disjoint cliques intersects in precisely one vertex, we get $\widetilde{k} = (m-1)(1+k/m)$. 

Now assume that $C_1$ and $C_2$ are distinct cliques in $\mathcal{C}$ that share a vertex $v$. Let $u\in C_1$ be a vertex distinct from $v$. Then $u$ has $\psi_1$ neighbors in $C_2$. Let $u'$ be one of such neighbors distinct from $v$. Then the edge $\{u, u'\}$ belongs to some clique $C$ which is distinct from $C_1$ and $C_2$ and intersects both of them. Thus, $C$ is a common neighbor of $C_1$ and $C_2$. Note, that a common neighbor $C\in \mathcal{C}$ of $C_1$ and $C_2$, which does not contain $v$, is uniquely determined  by $u\in C_1$ and its neighbor $u'$ in $C_2$. Finally, note that any clique from $\mathcal{C}$ which contains $v$ and is distinct from $C_1$ and $C_2$ is their common neighbor. Hence, $\widetilde{\lambda} = (m-2)+(\psi_1-1)k/m$.
\end{proof}

Let $A$ and $\widetilde{A}$ be the adjacency matrices of $X$ and $\widetilde{X}$. Denote by $\spec(X)$ and $\spec(\widetilde{X})$ the sets of eigenvalues of $A$ and $\widetilde{A}$, respectively.

\begin{lemma}\label{lem:dual-spec} Let $X$ be a geometric distance-regular graph with smallest eigenvalue $-m$. If $k\geq m^2$, then $\displaystyle{\spec(\widetilde{X})\subseteq \{ \theta-\frac{k}{m}+m-1\mid \theta \in \spec(X)\}}$.
 
\end{lemma} 
\begin{proof} Let $\mathcal{C}$ be a Delsarte clique geometry of $X$. Note that, by Lemma~\ref{lem:eigenvalue-clique-geom}, every vertex of $X$ belongs to precisely $m$ cliques of $\mathcal{C}$. Define $N$ to be an $n\times |\mathcal{C}|$ vertex-clique incidence matrix, i.e, $N_{i, j} = 1$ if the vertex $v_i$ belongs to the clique $C_j$, and $N_{i, j} = 0$ otherwise. Then
\begin{equation}\label{eq:spec}
 A = NN^{T} - mI \qquad \text{and} \qquad \widetilde{A} = N^TN - \left(\frac{k}{m}+1\right)I.
\end{equation}
From linear algebra it is known that non-zero eigenvalues of $NN^T$ and $N^TN$ coincide. Since $|\mathcal{C}|(1+k/m) = nm$, we get $|\mathcal{C}|< n$, so 0 is an eigenvalue of $NN^T$. Therefore, $\spec(N^TN)\subseteq \spec(NN^T)$ and the statement of the lemma follows from Eq.~\eqref{eq:spec}.   
\end{proof}

\section{Spectral gap characterization of Johnson graphs}\label{sec:Johnson}

In this section we prove Theorem~\ref{thm:pseudo-main-intr}, our characterization of Johnson graphs. Namely, we prove that a distance-regular graph with $\theta_1+1> (1-\varepsilon^*)b_1$ and connected neighborhood graphs is a Johnson graph (for sufficiently large $k$). We also show that the inequality $\theta_1+1> (1-\varepsilon^*)b_1$ can hold for a distance-regular graph with disconnected neighborhood graphs only if $\mu\leq 2$ (see Proposition~\ref{prop:mu-eigen-bound}).

%
%


The main idea of the proofs is to use the fact that for $\displaystyle{b^{+} = \frac{b_1}{\theta_1+1}}$ the expression $-1-b^{+}$ is a lower bound on the smallest eigenvalue of the neighborhood graph $X(v)$. More precisely, we use the following result of Terwilliger \cite{ter-local}.

\begin{theorem}[Terwilliger \cite{ter-local}, see {\cite[Theorem 4.4.3]{BCN}}]
Let $X$ be a distance-regular graph of diameter $d\geq 2$ with distinct eigenvalues $k = \theta_0>\theta_1>\ldots>\theta_d$, and put $\displaystyle{b^{+} = \frac{b_1}{\theta_1+1}}$, $\displaystyle{b^{-} = \frac{b_1}{\theta_d+1}}$. Then 
each neighborhood graph $X(v)$ has the smallest eigenvalue $\geq -1-b^+$, and the second largest eigenvalue $\leq -1-b^{-}$.
\end{theorem}

Recall, we assume that the second largest eigenvalue of $X$ satisfies $\theta_1+1\geq (1-\varepsilon)b_1$. In this case the smallest eigenvalue of the neighborhood graph $X(v)$ is at least $-2-\delta$, for $\delta = \varepsilon/(1-\varepsilon)$. We also observe that if $X$ is an edge-regular graph, its neighborhood graph $X(v)$ is regular for every vertex $v\in X$. 

First, we note that if the diameter $d$ of a distance-regular graph $X$ is at least 2, $\lambda>2$ and the neighborhood graph $X(v)$ is connected, then the smallest eigenvalue of $X(v)$ is at most $-2$. Indeed, if a regular connected graph has the smallest eigenvalue $>-2$, then it is a complete graph or an odd polygon. The neighborhood graph $X(v)$ cannot be complete as $d\geq 2$ and $X(v)$ cannot be an odd polygon as $\lambda>2$.

The graphs for which smallest eigenvalue is precisely $-2$ were classified by Cameron, Goethals, Seidel and Shult~\cite{regular-classif}. We use their classification in the case when a graph $Y$ is a connected regular graph.    

\begin{theorem}[Cameron et al. \cite{regular-classif}, see {\cite[Theorem 3.12.2]{BCN}}]
Let $Y$ be a connected regular graph with $n\geq 29$ vertices and smallest eigenvalue $\geq -2$. Then $Y$ is a line graph of a regular connected graph or of a bipartite semiregular connected graph.
\end{theorem}

Next, we use the results of Hoffman, Bussemaker and Neumaier, which assert that for a small enough $\delta$ the smallest eigenvalue of a graph is never in the interval $(-2-\delta, -2)$. 

\begin{definition}\label{def:eig_bounds} Define $\vartheta_k$ to be the supremum of the smallest eigenvalues of graphs with minimal valency $k$ and smallest eigenvalue $<-2$.
\end{definition}

\begin{theorem}[Hoffman \cite{Hoffman}, see {\cite[Theorem 3.12.5]{BCN}}]\label{thm:Hoffman}
The sequence $(\vartheta_k)_k$ forms a monotone decreasing sequence with limit $-1-\sqrt{2}$.
\end{theorem}

\begin{theorem}[Bussemaker, Neumaier \cite{Bussem-Neum}, see {\cite[Theorem 3.12.5]{BCN}}]\label{thm:theta1}
$\vartheta_1\approx -2.006594$ is the smallest root of the equation 
\[ \theta^2(\theta^2-1)^2(\theta^2-3)(\theta^2-4) = 1.\]
\end{theorem}

The above discussion can be summarized in the following proposition.

\begin{proposition}\label{prop:local-structure} Let $X$ be a distance-regular graph of diameter $d\geq 2$. Assume that the second largest eigenvalue of $X$ satisfies $\theta_1+1> (1-\varepsilon^{*})b_1$, for $\displaystyle{0<\varepsilon^* = \frac{-2-\vartheta_1}{-1-\vartheta_1}}$. Then for every vertex $v$ of $X$, the neighborhood graph $X(v)$ is a regular graph with smallest eigenvalue at least $-2$. 

Moreover, if $X(v)$ is connected, $\lambda>2$, and the vertex degree in $X$ is at least $29$, then $X(v)$ is the line graph of a regular or of a bipartite semiregular connected graph. 
\end{proposition}

\begin{remark}\label{rem:varepsilon-star} One can compute that $\varepsilon^{*} \approx 0.006551$. Observe that, the neighborhood graph $X(V)$ is regular of degree $\lambda$, and by Theorem~\ref{thm:Hoffman}, $\lim\limits_{\lambda \rightarrow\infty}\vartheta_{\lambda} = -1-\sqrt{2}$. Thus,  we can replace $\varepsilon^*$ with any number less than $1-1/\sqrt{2}\approx 0.29289$, if we additionally require $\lambda$ to be sufficiently large.
\end{remark}

Next we analyze the structure of the local graph $X(v)$ in the case when $X$ is geometric.

\begin{lemma}\label{lem:reduction-local-bipartite} Let $X$ be a geometric distance distance-regular graph with smallest eigenvalue $-m$.  Suppose that $X(v)$ is the line graph of a regular or a bipartite semiregular connected graph. Assume that vertex degree $k\geq \max(m^3, 3)$. Then, $X(v)$ is the line graph of a complete bipartite graph $K_{s, t}$ for each vertex $v$ of $X$, where $\{s, t\} = \{m, k/m\}$.
\end{lemma}
\begin{proof}
 Fix a Delsarte clique geometry $\mathcal{C}$ of $X$. Let $C_1, C_2, \ldots, C_m \in \mathcal{C}$ be the cliques that contain a vertex $v$. Since every edge of $X$ is contained in precisely one clique, every vertex of $N(v)$ is contained in precisely one of $C_1, C_2, \ldots, C_m$. Let $u\in C_1\setminus \{v\}$, by the definition of $\psi_i$ (see Section~\ref{sec:geometric}), $u$ is adjacent with precisely $\psi_1$ vertices of $C_i$ for all $i = 2, 3, \ldots, m$. Therefore, the degree of every vertex $u$ in $X(v)$ equals $k/m-1+(\psi_1-1)(m-1)$. 
 
 Assume that $Y$ is the line graph of a regular graph $Z$ with vertex degree $t$. Then the degree of a vertex in $Y$ is equal $2(t-1)$. Moreover, the size of a maximal clique in $Y$ is $t$, if $t\geq 3$. Since $Y$ contains a clique of size $k/m$, we get $t\geq k/m$. Therefore, if $k/m\geq 3$ and $(k/m-1)>(\psi_1-1)(m-1)$, then $X(v)$ is not a line graph of a regular graph. In particular this is true, if $k\geq \max(m^3, 3)$, as $\psi_1\leq \tau_2\leq m$ by the definition of $\tau_2$ and Lemma~\ref{lem:tau2}. 
 
  Hence, for every $v$, the neighborhood graph $X(v)$ is the line graph of a complete bipartite graph $K_{s, t}$. The size of the maximal clique in the line graph of $K_{s, t}$ is $\max(s, t)$. Thus, $\max(s, t) = k/m$. There are $k$ vertices in $X(v)$ and $st$ vertices in the line graph of $K_{s, t}$, so $\{s, t\} = \{m, k/m\}$.  
\end{proof}

In the case when $X(v)$ is the line graph of a complete bipartite graph $K_{s, t}$ and $1+\theta_1\geq (1-\varepsilon)b_1$ we show that $X$ is a Johnson graph. Our goal is to use the characterization of the Johnson graphs by local structure stated in Theorem~\ref{thm:Johnson-classif}. The only condition we still need to verify is $\mu \leq 4$. We prove that if $\mu>4$, then $X$ contains an induced subgraph $K_{3, 2}$ and so we can use the inequality provided by the theorem below.

\begin{theorem}[{\cite[Theorem 4.4.6]{BCN}}]\label{thm:induced-bipartite-ineq} Let $X$ be a distance-regular graph of diameter $d\geq 2$ with eigenvalues $k = \theta_0>\theta_1>\ldots>\theta_d$ and put $b^+ = b_1/(\theta_1+1)$. If $X$ contains a non-empty induced complete bipartite subgraph $K_{s, t}$, then 
\[ \frac{2st}{s+t}\leq b^{+}+1.\]  
\end{theorem}

\begin{lemma}\label{lem:mu-bound-psi1}
Let $X$ be a geometric distance-regular graph of diameter $d\geq 2$.
\begin{enumerate}
\item Assume that $\psi_1 = 1$, then $X$ contains an induced $K_{\tau_2, 2}$.
\item Assume that $\mu\geq 2$, then $X$ contains an induced $K_{2, 2}$ (a quadrangle).
\end{enumerate}  
\end{lemma}
\begin{proof}
Let $u$ and $v$ be two vertices at distance $2$ in $X$. By the definition of $\tau_2$ there exist distinct cliques $C_1, C_2, \ldots, C_{\tau_2}$ which contain $u$ and have non-trivial intersection with $N(v)$. 

\begin{enumerate}
\item Each $C_i$ has precisely $\psi_1 = 1$ common vertices with $N(v)$. 
Denote $w_i = C_i\cap N(v)$. Note that $w_i$ is at distance $1$ from $C_j$ for $i\neq j$, moreover, $w_i$ is adjacent to $u$, while $u\in C_j$ and $u\neq w_j$. Thus $w_j$ is not adjacent to $w_i$ for $i\neq j$. Therefore, $X$ contains an induced $K_{\tau_2, 2}$ (on vertices $\{w_1, w_2, \ldots, w_{\tau_2}, u, v\}$).
\item By Lemma~\ref{lem:tau2}, $\tau_2\geq 2$, if $\mu\geq 2$. Take $w\in N(v)\cap C_1$. Assume there are no induced $K_{2, 2}$ in $X$. Then $w$ is adjacent to each vertex in $T = C_2\cap N(v)$. Note that $|T| = \psi_1$, $u\notin T$ and $w$ is adjacent to $u$, so $w$ has at least $\psi_1+1$ neighbors in $C_2$. This gives a contradiction with the definition of $\psi_1$.
\end{enumerate}
\end{proof}

Using the lemma above we obtain the following corollary to the Theorem~\ref{thm:induced-bipartite-ineq}.

\begin{proposition}\label{prop:mu-eigen-bound} Let $X$ be a geometric distance-regular graph of diameter $d\geq 2$. Assume that the neighborhood graphs of $X$ are disconnected. If $\mu\geq 3$, then the second largest eigenvalue of $X$ satisfies $\displaystyle{\theta_1+1\leq \frac{5}{7} b_1}$.  
\end{proposition}
\begin{proof}
Since $X$ is geometric and $X(v)$ is disconnected, by Lemma~\ref{lem:connected-disconn}, $\psi_1 = 1$.  Moreover, if $\mu\geq 3$, by Lemma~\ref{lem:mu-bound-psi1}, there is an induced $K_{3, 2}$. Therefore, by Theorem~\ref{thm:induced-bipartite-ineq}, 
\[ \frac{b_1}{\theta_1+1}\geq \frac{12}{5}-1 = \frac{7}{5}.\]
\end{proof}

In the next lemma we show the existence of an induced complete bipartite subgraph $K_{\tau_2, 2}$ in the case when a neighborhood graph is the line graph of a triangle-free graph. 
\begin{lemma}\label{lem:local-line-graph} Let $X$ be a geometric distance-regular graph. Assume that for each vertex $v$ of $X$ the induced subgraph $X(v)$ is the line graph of a triangle-free graph. Then $\psi_1= 2$ and $X$ contains induced $K_{\tau_2, 2}$. 
\end{lemma}
\begin{proof}
Observe that if the line graph $Y$ of a triangle-free graph $Y'$ contains a triangle, then the corresponding edges (of the base graph $Y'$) of all three vertices of the triangle are incident to the same vertex in $Y'$.
  
Fix a Delsarte clique geometry $\mathcal{C}$ of $X$ . Let $v$ be a vertex of $X$, and $C\in \mathcal{C}$ be a Delsarte clique which contains $v$, and let $w\in N(v)\setminus C$. By Lemma~\ref{lem:mu-bound-psi1}, since $X(v)$ is connected, $\psi_1\geq 2$. Assume that $\psi_1\geq 3$. Then $w$ is adjacent to at least two vertices $v_1$ and $v_2$ in $C$ distinct from $v$. Since $w, v_1, v_2$ form a triangle in $X(v)$, the corresponding edges in the base graph are incident to the same vertex. Similarly, for any vertex $x\in C\setminus \{v\}$, the edges of the base graph that correspond to $x$, $v_1$ and $v_2$ are incident to the same vertex. Therefore, $\{w\}\cup C$ is a clique in $X$, which contradicts maximality of $C$. Therefore, $\psi_1 = 2$.
       
Let $u$ and $v$ be two vertices at distance $2$ in $X$. There exist distinct cliques $C_1, C_2, \ldots, C_{\tau_2}$ which contain $u$ and have non-trivial intersection with $N(v)$, and distinct cliques $C_1', C_2', \ldots, C_{\tau_2}'$ which contain $v$ and have non-trivial intersection with $N(u)$. 

Let $T = N(u)\cap N(v)$. Assume that $w_1, w_2\in T$ are adjacent, but $\{w_1, w_2\}$ is not a subset of $C_i$ or $C_i'$ for any $i\in [\tau_2]$. Since $\psi_1 = 2$ there exists a vertex $w\in T$ such that $\{w, w_1\}$ is a subset of some $C_i$. Similarly there is $w'\in T$ with $\{w', w_1\}\subseteq C_j'$ for some $j$. Assume that $X(v)$ is a line graph of a triangle-free graph $Y$. Assume that edges of $Y$ that correspond to $w_1, w'$ are incident with a vertex $x$ of $Y$. If corresponding to $w_2$ edge is incident with $x$, then by the argument as above, $w_2\in C_j'$. This contradicts the choice of $w_2$. Let $y$ be the vertex of $Y$ incident to the edges in $Y$ corresponding to $w_2$ and $w_1$. Since $|C_i\cap C_j'|\leq 1$, we have $w\notin C_j'$, so $w$ is not incident to $x$. Thus, $w$ is incident to $y$. Hence, $w$, $w_1$ and $w_2$ form a triangle.  

Since $\{w, w_1\}\subseteq C_i$ and $X(u)$ is also the line graph of a triangle-free graph we similarly get that $w_2\in C_i$. This gives a contradiction with the choice of $w_1, w_2$. Therefore, any two vertices in $T$ are adjacent if and only if they share the same clique $C_i$ or $C_i'$.

We obtain that an edge between a pair of vertices in $T$ is between vertices in $T\cap C_i$ or between vertices in $T\cap C_j'$ for some $i, j$. We refer to them as edges of type 1 and edges of type 2, respectively. Since $|T\cap C_i| = |T\cap C_i'| = \psi_1 = 2$, every vertex in $T$ is incident with precisely one edge of type 1 and precisely one edge of type 2. Therefore, the subgraph induced on $T$ is a union of even cycles. Hence, we may choose an independent set $S\subseteq T$ of the size $\frac{1}{2}|T| = \frac{1}{2}\psi_1\tau_2  = \tau_2$ in $X$. The graph induced on $S\cup\{u, v\}$ is $K_{2, \tau_2}$.
\end{proof}

Now we are ready to combine all the arguments into the proof of Theorem \ref{thm:pseudo-main-intr}.

\begin{proof}[Proof of Theorem \ref{thm:pseudo-main-intr}] Denote $\displaystyle{b^{+} = \frac{b_1}{\theta_1+1}}$. Since $\displaystyle{\varepsilon^* = \frac{-2-\vartheta_1}{-1-\vartheta_1}}$ the assumptions of the theorem imply $b^+<-1-\vartheta_1$. Since $X$ is not complete, by Lemma~\ref{lem:eigenvalue-clique-geom}, $m\geq 2$. Thus $k\geq m^3$ and $\lambda\geq k/m-1$ implies $\lambda\geq 3$. Hence,  by Proposition~\ref{prop:local-structure} and Lemma \ref{lem:reduction-local-bipartite}, either $X(v)$ is a disconnected graph for some vertex $v$, or $X(v)$ is the line graph of the complete bipartite graph $K_{s, t}$ for any vertex $v$ of $X$. 

In the latter case, by Lemma~\ref{lem:local-line-graph}, $X$ contains $K_{\tau_2, 2}$ and $\psi_1 = 2$. Since we assumed that $b^{+}\leq -1-\vartheta_1 < 7/5$, by Theorem~\ref{thm:induced-bipartite-ineq}, we get that $\tau_2\leq 2$ (as otherwise there is induced $K_{2, 3}$ subgraph). Hence, $\mu\leq 4$. By Theorem~\ref{thm:Johnson-classif}, we get that $X$ is a Johnson graph $J(s, d)$, or a graph which can be double covered by $J(2d, d)$. The later case is not possible as $k\geq m^3$.

\end{proof}
\begin{remark}\label{rem:varepsilon27} We note that in the light of Remark~\ref{rem:varepsilon-star}, if we additionally assume that $k$ is large enough, then we can replace $\varepsilon^*$ with $2/7$ in Theorem~\ref{thm:pseudo-main-intr}. Indeed, since $\lambda\geq (k/m)-1$ and $k\geq m^3$, the assumption that $k$ is large enough guarantees that $\lambda$ is large enough. Hence, the proof above will work as $b^{+}<7/5<\sqrt{2}$. 
\end{remark}

\section{Spectral gap characterization of Hamming graphs}\label{sec:Hamming}

In this section we prove a characterization of Hamming graphs in terms of the spectral gap and local parameters. As in the previous section, we assume that the second largest eigenvalue of $X$ satisfies $\theta_1\geq  (1-\varepsilon)b_1$. We show that if additionally for each vertex the neighborhood graph is a disjoint union of cliques, $\mu = 2$, and there is a dominant distance, then $X$ is a Hamming graph. Our proof strategy relies on the following theorem.  

\begin{theorem}[Terwilliger \cite{ter-local}]\label{thm:ensure-eigenvalue} Let $X$ be a distance-regular graph of diameter $d\geq 2$. Assume that the second largest eigenvalue $\theta_1$ has multiplicity $f_1<k$. Then each local graph $X(v)$ has eigenvalue $-1-b^{+}$ with multiplicity at least $k-f_1$, where $\displaystyle{b^{+} = \frac{b_1}{\theta_1+1}}$.
\end{theorem}

We prove that if $X$ is not a Hamming graph, then the assumptions of Theorem~\ref{thm:main-Hamming-intr} imply that the second largest eigenvalue has multiplicity at most $k-1$. Therefore, by the theorem above, each neighborhood graph of $X$ has an eigenvalue less than $-1$. This contradicts the assumption that each neighborhood graph is a disjoint union of cliques. 

We start by showing that for a geometric distance-regular graph the sequence $(\tau_i)_{i=1}^{t-1}$ is increasing if $\mu\geq 2$ and $c_t$ is sufficiently small.

\begin{lemma}\label{lem:tau-ineq}
Let $X$ be a geometric distance-regular graph of diameter $d$, with smallest eigenvalue $-m$. Assume that $\mu\geq 2$ and $c_t\leq \varepsilon k$, where $t\leq d$ and $0<\varepsilon <1/m^2$. Then 
\[\tau_i<\tau_{i+1},\quad  \text{for any } i = 1, 2, \ldots, t-2.\]
\end{lemma}
\begin{proof}

Recall, by Lemma~\ref{lem:geometric-param},
\[ c_i = \tau_i\psi_{i-1}, \quad \quad \displaystyle{b_i = (m-\tau_i)\left(\frac{k}{m}+1-\psi_i\right)}.\]
Hence, in particular,  $\psi_{i-1}\leq c_i\leq c_t\leq \varepsilon k$, for $i\leq t$. So for $i\leq t-1$
\begin{equation}\label{eq:b-ineq}
 (m-\tau_i)\left(\frac{1}{m}-\varepsilon\right)k\leq b_i\leq \frac{m-\tau_i}{m}k.
\end{equation} 
By Lemma~\ref{lem:mu-bound-psi1}, a geometric distance-regular graph with $\mu\geq 2$ contains a quadrangle. Thus, by the Terwilliger inequality (see Theorem~\ref{thm:ter-inequality}) we have 
 \[ b_{i}\geq b_{i+1}+\lambda+2+c_{i}-c_{i+1}, \text{ for } i = 0, 1, \ldots, d-1.\]
Therefore, for $i\leq t-2$, using Eq.~\eqref{eq:b-ineq}, 
\[ \frac{m-\tau_i}{m}k\geq (m-\tau_{i+1})\left(\frac{1}{m}-\varepsilon\right)k+\lambda+2-\varepsilon k.\]
Since $\lambda\geq k/m-1$, for $i\leq t-2$, we get
\[ (m-\tau_i)\geq (m-\tau_{i+1})-m^2\varepsilon+1 \quad \Rightarrow \quad \tau_{i+1}\geq \tau_i+1-m^2\varepsilon.\]
\end{proof}

\begin{corollary} If the assumptions of Lemma~\ref{lem:tau-ineq} hold for $t=d$, then $\tau_i< \tau_{i+1}$ for $i\leq d-1$.
\end{corollary}
\begin{proof}
By Lemma~\ref{lem:tau-ineq}, $\tau_i< \tau_{i+1}$ for $i\leq d-2$. Observe, that by the definition of $\tau_i$, we have $\tau_d = m$ and $\tau_{i}\leq m-1$ for any $i\leq d-1$.
\end{proof}

\begin{corollary}\label{cor:b-ineq} If the assumptions of Lemma~\ref{lem:tau-ineq} hold for $t=d$, then
\[ (d-i)\left(\frac{1}{m}-\varepsilon\right)k\leq b_i\leq \frac{m-i}{m}k,\quad \text{for } 1\leq i\leq d-1.\]
\end{corollary}
\begin{proof}
Since $\tau_1 = 1$ and $\tau_{d-1} \leq m-1$, by Lemma \ref{lem:tau-ineq}, we have $i\leq \tau_i\leq m-d+i$ for $i\leq d-1$. So the desired inequality directly follows from Eq.~\eqref{eq:b-ineq}.
\end{proof}

To get a bound on the multiplicity of the second largest eigenvalue $\theta_1$ of $X$ we first prove lower bounds on the elements of the standard sequence corresponding to $\theta_1$ (see Sec.~\ref{sec:drg}).

\begin{lemma}\label{lem:stand-seq-ineq} Let $X$ be a geometric distance-regular graph of diameter $d\geq 2$ with smallest eigenvalue $-m$. Let $\theta_1$ be its second largest eigenvalue and $(u_i)_{i=0}^{d}$ be the corresponding standard sequence. Assume that $\mu\geq 2$, $\theta_1\geq(1-\varepsilon)b_1$, and $c_t\leq \varepsilon k$ for some $2\leq t\leq d$, where $0<\varepsilon<1/(24m^2)$. 

Then, for $1\leq j\leq t-1$ 
\[u_j\geq (1-3m^2\varepsilon)^{j-1}\left(\frac{m-\tau_{j}}{m-\tau_j+j-1}\right)\frac{\theta_1}{k}.\]

\end{lemma}
\begin{proof}
Recall that the standard sequence corresponding to the eigenvalue $\theta_1$ satisfies 
\[ u_0 = 1, \qquad u_1 = \frac{\theta_1}{k}, \qquad c_iu_{i-1}+a_iu_i+b_iu_{i+1} = \theta_1 u_i, \text{ for } i = 1, \ldots,  d-1.\]
We can rewrite this as
\begin{equation}\label{eq:u-recurrence}
 u_{i+1} = u_i\frac{(\theta_1+b_i+c_i-k)}{b_i} - u_{i-1}\frac{c_i}{b_i}\geq u_i\left(1-\frac{k-\theta_1}{b_i}\right)- u_{i-1}\frac{c_i}{b_i}.
 \end{equation}
For $2\leq i\leq t$, by the assumptions of the lemma,  $\psi_{i-1}\leq c_i\leq c_t\leq \varepsilon k$. So, by Lemma~\ref{lem:geometric-param}, 
\begin{equation}\label{eq:ktheta}
k-\theta_1\leq k-(1-\varepsilon)b_1\leq (k-b_1)+\varepsilon k\leq \frac{k}{m}+(m-1)\psi_1+\varepsilon k\leq \frac{k}{m}+m\varepsilon k;
\end{equation}
\begin{equation}\label{eq:blowerbound}
  (m-\tau_i)\left(\frac{1}{m}-\varepsilon\right)k \leq b_i,\quad \text{for } i\leq t-1. 
  \end{equation}
For the convenience of the future computations we first show that for $1\leq i\leq t-2$, inequality $3u_{i+1}\geq u_{i}\geq 0$ holds. Indeed, by Eq.~\eqref{eq:ktheta}, $u_1\geq \frac{1}{3}u_0$. Moreover, by Lemma~\ref{lem:tau-ineq}, $\tau_i\leq \tau_{t-1}-1\leq m-2$, so by Eq.~\eqref{eq:ktheta} and Eq.~\eqref{eq:blowerbound},
\[ 1 - \frac{k-\theta_1}{b_i}\geq 1 - \frac{1+m^2\varepsilon}{2-2m\varepsilon} \geq \frac{1}{2} - m^2\varepsilon.\]
Thus, using $\tau_i\leq m-2$ for $i\leq t-2$, by induction, we get from Eq.~\eqref{eq:u-recurrence} and Eq.~\eqref{eq:blowerbound}
\[u_{i+1}\geq \left(\frac{1}{2}-m^2\varepsilon\right)u_i - m\varepsilon u_{i-1}\geq \left(\frac{1}{2}-4m^2\varepsilon\right)u_i\geq \frac{1}{3}u_i.\]
Hence, for $i\leq t-2$, we can rewrite Eq.~\eqref{eq:u-recurrence}
\[ u_{i+1} \geq  u_i\frac{(\theta_1+b_i+c_i-k)}{b_i} - 3u_i\frac{c_i}{b_i}\geq u_i\left(1 - \frac{k-\theta_1+2\varepsilon k}{b_i}\right).\]
Thus, using Eq.~\eqref{eq:ktheta} and Eq.~\eqref{eq:blowerbound}, for $i\leq t-2$,
\[ u_{i+1}\geq u_i\left(1 - \frac{k+(m^2+2m)\varepsilon k}{mb_i}\right)\geq u_i\left(1 - \frac{1}{(m-\tau_i)}\frac{(1+2m^2\varepsilon)}{(1-m\varepsilon)}\right)\geq u_i\left(1-\frac{(1+3m^2\varepsilon)}{m-\tau_i}\right).\]
By Lemma~\ref{lem:tau-ineq}, $\tau_i\leq \tau_{j} - (j-i)$ for $i\leq j\leq t-1$. Thus, for $\delta = 3m^2\varepsilon$ and $i+1\leq j\leq t-1$, 
\[ u_{i+1}\geq (1-\delta)\left(1-\frac{1}{m-\tau_i}\right)u_i\geq (1-\delta)\frac{m-\tau_{j}+j-i-1}{m-\tau_{j}+j-i}u_i.\]
Therefore, for any $1\leq j\leq t-1$,
\[ u_{j}\geq (1-\delta)^{j-1}\prod\limits_{i=1}^{j-1}\frac{m-\tau_{j}+j-i-1}{m-\tau_{j}+j-i}u_1 =  (1-\delta)^{j-1}\left(\frac{m-\tau_{j}}{m-\tau_j+j-1}\right)\frac{\theta_1}{k}. \] 
\end{proof}

\begin{proposition}\label{prop:multiplicity} Let $X$ be a geometric distance-regular graph of diameter $d\geq 2$ with smallest eigenvalue $-m$. Take  
$\displaystyle{0<\varepsilon< 1/(6m^4d)}$.
Suppose that $\mu\geq 2$, $c_t\leq \varepsilon k$ and $b_t\leq \varepsilon k$ for some $2\leq t\leq d$. Assume, moreover, that the second largest eigenvalue of $X$ satisfies 
$\theta_1\geq (1-\varepsilon)b_1$.  

Then either the multiplicity $f_1$ of $\theta_1$ satisfies $f_1 \leq k-1$, or $m = d$, $t=d$ and $c_d = d$. 
\end{proposition}
\begin{proof} Let $(u_i)_{i=0}^{d}$ be the standard sequence of $X$ corresponding to $\theta_1$.
Then, by the Biggs formula, multiplicity of an eigenvalue $\theta_1$ can be computed as
\[ f_1 = \frac{n}{\left(\sum\limits_{i=0}^{d}k_iu_i^{2}\right)}.\]
Note, as in Eq.~\eqref{eq:b-ineq}, $\displaystyle{b_{i-1}\geq b_{t-1}\geq \left(\frac{1}{m}-\varepsilon\right)k}\geq \frac{1}{2m}k$ and $c_i\leq c_t$ for any $1\leq i\leq t$. So
\begin{equation}\label{eq:degree-before-t} k_{i-1} = \frac{c_{i}}{b_{i-1}}k_i\leq \frac{c_t}{b_{t-1}}k_{i} \leq 2m\varepsilon k_i, \quad \text{for } i\leq t.
\end{equation}
For $d-1\geq i\geq t$, by Lemma~\ref{lem:geometric-param}, $b_i\leq b_t\leq \varepsilon k$ implies 
\[\displaystyle{\psi_i\geq \left(\frac{1}{m}-\varepsilon\right) k\geq \frac{1}{2m}k},\qquad \text{so} \qquad \displaystyle{c_{i+1} = \tau_{i+1}\psi_{i}\geq \frac{1}{2m}k}.\]
Hence, for $t\leq i\leq d-1$ we deduce,
\begin{equation}\label{eq:degree-after-t}
 k_{i+1} = \frac{b_{i}}{c_{i+1}}k_i\leq \frac{\varepsilon k}{k/(2m)}k_i = 2m\varepsilon k_i.
\end{equation}
Combining Eq.~\eqref{eq:degree-before-t} and Eq.~\eqref{eq:degree-after-t} we obtain
\[ n = \sum\limits_{i=0}^{d} k_i \leq k_t\left(\sum\limits_{i=0}^{t}(2m\varepsilon)^{i}+\sum\limits_{i=1}^{d-t}(2m\varepsilon)^{i}\right)\leq \frac{1}{1-4m\varepsilon}k_t \qquad \Rightarrow \qquad k_t\geq (1-4m\varepsilon)n.\]
As in Eq.~\eqref{eq:ktheta}, $\theta_1/k\geq (m-1)/m-m\varepsilon$. So, by Lemma~\ref{lem:stand-seq-ineq} and Eq.\eqref{eq:b-ineq}, for $t\geq 2$,
\[ k_{t-1}u_{t-1}^{2}\geq k_t\frac{c_t}{b_{t-1}}(1-3m^2\varepsilon)^{2t-4}\left(\frac{m-\tau_{t-1}}{m-\tau_{t-1}+t-2}\right)^2 \left(\frac{\theta_1}{k}\right)^2\geq\]
\[ \geq k_t\frac{c_t}{b_{t-1}}(1-3m^2\varepsilon)^{2t-4}\left(\frac{m-\tau_{t-1}}{m-\tau_{t-1}+t-2}\right)^2(1-2m\varepsilon)^2\left(\frac{m-1}{m}\right)^2\geq \]
\[ \geq(1-4m\varepsilon)n\cdot \frac{m c_t}{(m-\tau_{t-1})k }\cdot (1-3m^2\varepsilon)^{2d-1}  \left(\frac{m-\tau_{t-1}}{m-\tau_{t-1}+t-2}\right)^2\left(\frac{m-1}{m}\right)^2 \geq \]
\[ \geq \frac{n}{k}\cdot(1-3m^2\varepsilon)^{2d} \cdot\frac{c_t}{m}\cdot\frac{(m-\tau_{t-1})(m-1)^2}{(m-\tau_{t-1}+t-2)^2}.\]

Our goal is to deduce from this inequality that $k_{t-1}u_{t-1}^{2}> n/k$, unless $c_t = t = m = d$. We start by giving a bound on $c_t$.
Observe that $\psi_{t-2}\geq 1$, and $\tau_{t-1} \geq t-1$, by Lemma~\ref{lem:tau-ineq}. So we obtain
  \begin{equation}\label{eq:ctbound}
  c_t \geq c_{t-1} \geq \tau_{t-1}\psi_{t-2}\geq t-1.
  \end{equation}
  
\noindent\textbf{Case 1.} First, assume that $c_t = t-1$. \\
Then, Eq.~\eqref{eq:ctbound} implies $\tau_{t-1} = t-1$ and $c_t = c_{t-1}$. Thus, in particular, we can simplify 
\[ \frac{c_t}{m}\cdot\frac{(m-\tau_{t-1})(m-1)^2}{(m-\tau_{t-1}+t-2)^2} = \frac{t-1}{m}\cdot \frac{(m-t+1)(m-1)^2}{(m-1)^2} = \frac{(t-1)(m-t+1)}{m}.\]

Also, observe that the constraint $c_t = c_{t-1}$ implies $t>2$ as $1=c_{1}<2\leq \mu = c_2$. Moreover, by Corollary~\ref{cor:c3-mu}, $c_3>c_2$ for $\mu\geq 2$, so we should have $t\geq 4$ in this case.

At the same time, $c_t = c_{t-1}$, using Terwilliger's inequality (see Theorem~\ref{thm:ter-inequality}), implies
\[ b_{t-1}\geq c_{t-1}-c_{t}+b_t+\lambda+2\geq \lambda+2\geq \frac{k}{m}+1.\]
Since
\[ b_{t-1} = (m-\tau_{t-1})\left(\frac{k}{m}+1-\psi_{t-1}\right)\leq (m-\tau_{t-1})\frac{k}{m},\]
we deduce $\tau_{t-1}\leq m-2$. So $t \leq m-1$, as $\tau_{t-1} = t-1$. 

Thus $4\leq t\leq m-1$, which implies $m\geq 5$, and we get
\[ \frac{c_t}{m}\cdot\frac{(m-\tau_{t-1})(m-1)^2}{(m-\tau_{t-1}+t-2)^2} = \frac{(t-1)(m-t+1)}{m}\geq \frac{2(m-2)}{m}\geq 2-\frac{4}{m}\geq \frac{6}{5}.\]
Since, $3m^2\varepsilon<1$, by Bernoulli's inequality
\[ (1-3m^2\varepsilon)^{2d}\geq (1-6dm^2\varepsilon)>\frac{5}{6}.\]
Therefore, in this case,
\[ k_{t-1}u_{t-1}^2 > \frac{n}{k} \quad \Rightarrow\quad f_1\leq \frac{n}{k_{t-1}u_{t-1}^2}<k \qquad \Rightarrow \qquad f_1\leq k-1.\]

\noindent\textbf{Case 2.} Else we have $c_t\geq t$.\\ 
Lemma \ref{lem:tau-ineq} implies $t\leq \tau_{t-1}+1\leq m$. It is not hard to check (see Appendix \ref{appendix}), that 
\begin{equation}\label{eq:simplify}
  \frac{(m-\tau_{t-1})(m-1)^2}{(m-\tau_{t-1}+t-2)^2} \geq \frac{m-1}{t-1}. 
\end{equation}
Hence, applying inequality from Eq.~\eqref{eq:simplify},
\[  k_{t-1}u_{t-1}^2\geq \frac{n}{k}(1-3m^2\varepsilon)^{2d}\frac{c_t}{m}\left(\frac{m-1}{t-1}\right)\geq \frac{n}{k}(1-3m^2\varepsilon)^{2d}\left(\frac{t}{t-1}\right)\left(\frac{m-1}{m}\right).\]
 If $2\leq t<m$, then 
 $$\displaystyle{\left(\frac{t}{t-1}\right)\left(\frac{m-1}{m}\right)\geq \left(\frac{m-1}{m-2}\right)\left(\frac{m-1}{m}\right)= 1+\frac{1}{m^2-2m}\geq 1+\frac{1}{m^2-1}}.$$  
 If $2\leq t = m$, and $c_t>t$, then 
 $$\displaystyle{\left(\frac{c_t}{t-1}\right)\left(\frac{m-1}{m}\right)\geq \left(\frac{t+1}{t-1}\right)\left(\frac{m-1}{m}\right) = \frac{m+1}{m}\geq 1+\frac{1}{m^2-1}}.$$ In each of these two cases, we get

\[ k_{t-1}u_{t-1}^2\geq \frac{n}{k}\cdot(1-3m^2\varepsilon)^{2d}\left(1+\frac{1}{m^2-1}\right).\]
By Bernoulli's inequality, since $3m^2\varepsilon<1$,
\[ (1-3m^2\varepsilon)^{2d}\geq (1-6dm^2\varepsilon)> 1-\frac{1}{m^2} = \left(1+\frac{1}{m^2-1}\right)^{-1}.\]
Therefore, if $c_t>t$ or $m>t$, then $k_{t-1}u_{t-1}^2>n/k$ and so   
\[f_1\leq \frac{n}{k_{t-1}u_{t-1}^2}<k \qquad \Rightarrow \qquad f_1\leq k-1.\] 
Finally, assume $\tau_t\psi_{t-1} = c_t = t$ and $m = t$. We know from Lemma~\ref{lem:tau-ineq} that $\tau_{t-1}\geq t-1$. If $t<d$, then $b_t\geq 1$. So, by Terwilliger's inequality and Lemma~\ref{lem:geometric-param},
\[ \frac{k}{m}\geq (m-\tau_{t-1})\left(\frac{k}{m}+1-\psi_{t-1}\right) = b_{t-1}\geq 
c_{t-1}-c_{t}+b_t+\lambda+2\geq \lambda+2\geq \frac{k}{m}+1,\]
which gives a contradiction with the assumption $t<d$. Therefore, $t=d$ and $c_d = d$, $m=d$.
\end{proof}

Now we are ready to prove Theorem~\ref{thm:main-Hamming-intr} in the following equivalent form (see Lemma~\ref{lem:connected-disconn}).
\begin{theorem}\label{thm:main-Hamming} Let $X$ be a distance-regular graph of diameter $d\geq 2$. Suppose that every neighborhood graph $X(v)$ is a disjoint union of $m$ cliques. Moreover, assume $\mu\geq 2$,  $c_t\leq \varepsilon k$ and $b_t\leq \varepsilon k$ for some $t\leq d$ and $\theta_{1}\geq (1-\varepsilon)b_1$, where $0<\varepsilon < 1/(6m^4d)$. 

Then $X$ is a Hamming graph $H(d, s)$, for $s = 1+k/d$.
\end{theorem}
\begin{proof} Pick some vertex $v$ and let $X(v) = \bigcup_{i=1}^{m} C_i$, where $C_i$ is a clique for every $i$. Since $X$ is distance-regular, all $C_i$ are of the same size $\lambda+1$. Note that $\{v\}\cup C_i$ is a maximal clique in $X$ of the size $k/m+1$. Since $k\geq 1/\varepsilon>m^2$,  by Proposition~\ref{suff-cond},  $X$ is geometric with smallest eigenvalue $-m$.

Hence, by Proposition~\ref{prop:multiplicity}, either we have $f_1\leq k-1$, or $c_d = m = d$. If $f_1\leq k-1$, by Theorem~\ref{thm:ensure-eigenvalue}, $\displaystyle{-1-\frac{b_1}{\theta_1+1}}$ is an eigenvalue of $X(v)$. However, $b_1>0$ and $\theta_1>0$, so $X(v)$ has an eigenvalue less than $-1$. This gives a contradiction with the assumption that $X(v)$ is a disjoint union of cliques.

Therefore, $c_d = m = d$. By Lemma~\ref{lem:tau-ineq}, we get $\tau_i \geq i$ for any $i\in [d]$. At the same time, $d = c_d = \tau_d\psi_{d-1}$, so $\tau_d = d$ and $\psi_{d-1} = 1$. We immediately deduce $\tau_i = i$ for all $i\in [d]$. Assume that $\psi_{i-1}\geq 2$, while $\psi_{i} = 1$ for some $2\leq i\leq d-1$. Then we get a contradiction with 
\[i+1 = \psi_{i}\tau_{i+1} = c_{i+1}\geq c_{i} = \psi_{i-1}\tau_i\geq 2i.\]  
Thus, $\psi_i = 1$ for every $i$. This means, that the intersection array of $X$ coincides with the intersection array of the Hamming graph $H(d, 1+k/d)$, as
\[ c_i = i \qquad \text{and} \qquad b_i = (d-i)\frac{k}{d}.\]
Using characterization of Hamming graphs by their intersection array (Theorem~\ref{thm:Hamming-array-class}), $X$ is a Hamming graph or a Doob graph. Note that $X$ may be a Doob graph, only if $1+k/d = 4$, which is not possible as $k\geq 1/\varepsilon\geq 6d$. Therefore, $X$ is a Hamming graph. 
\end{proof}

\begin{corollary}\label{cor:Hamming} Let $X$ be a geometric distance-regular graph of diameter $d\geq 2$. Suppose that $X$ has $\mu =2$ and smallest eigenvalue $-m$. Take  $0<\varepsilon <  1/(6m^4d)$. Assume, $c_t\leq \varepsilon k$ and $b_t\leq \varepsilon k$ for some $t\in [d]$, and $\theta_{1}\geq (1-\varepsilon)b_1$. Then $X$ is a Hamming graph $H(d, s)$.
\end{corollary}
\begin{proof} By Lemma~\ref{lem:geometric-param}, $\mu = \tau_2\psi_1$, and by Lemma~\ref{lem:tau2}, $\tau_2\geq \psi_1$, so $\mu = 2$ implies $\tau_2 = 2$ and $\psi_1 = 1$. This means that a neighborhood graph of $X$ is a disjoint union of $-\theta_d = m$ cliques (Lemma~\ref{lem:connected-disconn}). Therefore, the statement follows from the theorem above.

\end{proof}

\section{Proof of Babai's conjecture}

In this section we prove Babai's conjecture on the motion of distance-regular graphs. We start our discussion with the elusive case $\mu=1$.
\subsection{Geometric distance-regular graphs with $\mu = 1$}

In the case of $\mu=1$ our strategy is to show that the dual graph of $X$ has linear motion and then to deduce that $X$ itself has linear motion. We use the following Spectral tool for obtaining motion lower bounds, proved by Babai in~\cite{babai-srg}.

For a $k$-regular graph $X$ let $k = \xi_1\geq \xi_2\geq...\geq \xi_n$ denote the eigenvalues of the adjacency matrix of $X$. We call quantity $\xi = \xi(X) = \max\{ |\xi_i|: 2\leq i\leq n\}$ the \textit{zero-weight spectral radius} of $X$. 

\begin{lemma}[Babai, {\cite[Proposition~12]{babai-srg}}]\label{mixing-lemma-tool}
Let $X$ be a regular graph of degree $k$ on $n$ vertices with the zero-weight spectral radius $\xi$. Suppose every pair of vertices in $X$ has at most $q$ common neighbors. Then 
\[\motion(X)\geq n\cdot \frac{(k-\xi-q)}{k}.\]
\end{lemma}

%

Let $X$ be a geometric distance-regular graph with $\mu = 1$ and let $\widetilde{X}$ be its dual graph. By Lemma~\ref{lem:dual-grapp-degree}, every vertex in  $\widetilde{X}$ has degree $\displaystyle{\widetilde{k} = k\frac{m-1}{m}+(m-1)}$ and every pair of adjacent vertices in $\widetilde{X}$ has $\widetilde{\lambda} = m-2$ common neighbors. Every pair of vertices at distance two in $\widetilde{X}$ has $\widetilde{\mu} = 1$ common neighbours. Indeed, if there are at least two edges between a pair of cliques $C_1$ and $C_2$, that do not share a vertex, then either $\psi_1 \geq 2$ in $X$, or there is an induced quadrangle. In both cases we get $\mu\geq 2$ and we reach a contradiction.

Since $q(\widetilde{X}) = \max(\widetilde{\mu}, \widetilde{\lambda})$ is small, we are going to show that Lemma~\ref{mixing-lemma-tool} can be applied. For this, it is sufficient to show that $\widetilde{X}$ has a linear in $k$ spectral gap. First, we  bound the zero-weight spectral radius of a geometric distance-regular graph using the following lemma proven in \cite{kivva-drg}. 

\begin{theorem}[{\cite[Theorem 3.11]{kivva-drg}}]\label{thm:cited-main-spectral-gap}
 For every $d\geq 2$ there exist $\varepsilon = \varepsilon(d)>0$ and $\eta = \eta(d)>0$ such that for any  distance-regular graph $X$ of diameter $d$ one of the following is true.
\begin{enumerate}
\item For some $0\leq i\leq d-1$ we have $b_i\geq \varepsilon k$ and $c_{i+1}\geq \varepsilon k$.
\item The zero-weight spectral radius of $X$ satisfies 
$ \xi\leq k(1-\eta)$.
\end{enumerate} 
\end{theorem}

%

Using the relationship between the spectrum of the geometric graph $X$ and its dual graph $\widetilde{X}$ we get the following corollary.

\begin{lemma}
Let $X$ be a geometric distance-regular graph of diameter $d\geq 2$ with smallest eigenvalue $-m$, where $m\geq 3$. Let $\widetilde{X}$ be its dual graph. Let $\varepsilon=\varepsilon(d)$ and $\eta = \eta(d)\leq 1/2$ be constants provided by Theorem~\ref{thm:cited-main-spectral-gap}. Assume $k\geq m^2$, $c_t\leq \varepsilon k$ and $b_t\leq \varepsilon k$ for some $t\in [d]$.  Then the zero-weight spectral radius of $\widetilde{X}$ satisfies \[ \xi(\widetilde{X}) \leq \widetilde{k}(1-\eta).\] 

\end{lemma}
\begin{proof} Assumption $k\geq m^2$ implies $\widetilde{k}<k$. Let $\widetilde{\theta}_1$ and $\widetilde{\theta}_{\min}$ denote the second largest and the smallest eigenvalues of $\widetilde{X}$. Then the statement of the lemma follows from the following two inequalities implied by Theorem~\ref{thm:cited-main-spectral-gap} and Lemma~\ref{lem:dual-spec},

 \[  \widetilde{\theta}_1\leq (1-\eta)k-\frac{k}{m}+m-1= \widetilde{k}-\eta k\leq \widetilde{k}(1-\eta),\] 
\[ \widetilde{\theta}_{\min}\geq -m-\frac{k}{m}+m-1 = -\frac{k}{m}-1= -\frac{\widetilde{k}}{m-1}\geq -\widetilde{k}(1-\eta). \]
\end{proof}

Thus, using Lemma~\ref{mixing-lemma-tool} we get a linear lower bound on the motion of $\widetilde{X}$.

\begin{proposition}\label{prop:motion-dual} Let $X$ be a geometric distance-regular graph of diameter $d\geq 2$ with $\mu =1$ and smallest eigenvalue $-m$, where $m\geq 3$. Let $\varepsilon=\varepsilon(d)$ and $\eta = \eta(d)\leq 1/2$ be constants provided by Theorem~\ref{thm:cited-main-spectral-gap}.  Assume $c_t\leq \varepsilon k$ and $b_t\leq \varepsilon k$ for some $t\in [d]$, and $k\geq \max(4m/\eta, m^2)$. Let $\widetilde{X}$ be the dual graph of $X$. Then 
\[ \motion(\widetilde{X})\geq \frac{\eta}{2} |V(\widetilde{X})|.\]
\end{proposition}
\begin{proof}
Since $\mu=1$, by the discussion after Lemma~\ref{mixing-lemma-tool}, the maximal number of common neighbors of a pair of distinct vertices is equal $q(\widetilde{X}) = \max(\widetilde{\lambda}, \widetilde{\mu}) = \max(m-2, 1)$. Note that $\displaystyle{\widetilde{k} \geq \frac{m-1}{m} k\geq \frac{k}{2}}$. So $\eta \widetilde{k}\geq 2m\geq 2q(\widetilde{X})$. Hence, 
\[\xi(\widetilde{X})+q(\widetilde{X})\leq (1-\eta)\widetilde{k}+\frac{\eta}{2} \widetilde{k} = \left(1-\frac{\eta}{2}\right)\widetilde{k}.\]
Therefore, the statement of the proposition follows from Lemma~\ref{mixing-lemma-tool}.
\end{proof}

We show that this implies that $\motion(X)$ is linear in $n = |V(X)|$.

\begin{lemma}
Let $\mathcal{F}$ be a collection of size-$s$ subsets of the set $\Omega$ such that every element of $\Omega$ is in $m$ sets from $\mathcal{F}$ and any two distinct sets in $\mathcal{F}$ intersect in precisely one element of $\Omega$. Let $\sigma$ be a permutation of $\Omega$ which respects $\mathcal{F}$, namely for every $C\in \mathcal{F}$ its image $\sigma(C)$ is in $\mathcal{F}$, too. Assume that at most $\alpha|\mathcal{F}|$ sets $C\in \mathcal{F}$ are fixed by $\sigma$, then at most $\displaystyle{\left(\alpha+\frac{1-\alpha}{s}\right)|\Omega|}$ elements of $\Omega$ are fixed by $\sigma$.
\end{lemma}
\begin{proof}
Note that if $C\in \mathcal{F}$ is not fixed as set by $\sigma$, then $|\sigma(C)\cap C|\leq 1$, as $\sigma(C)\in \mathcal{F}$ too. Hence, at most one element $x\in \Omega$ of $C$ is fixed by $\sigma$. 

Now let us count the number of pairs $(C, v)$, such that $v\in C$ and $\sigma(v)\neq v$. We just argued that each of at least $(1-\alpha)|\mathcal{F}|$ sets in $\mathcal{F}$ have $(s-1)$ elements that are not fixed by $\sigma$. Therefore, the number of pairs we count is at least $(1-\alpha)|\mathcal{F}|(s-1)$. Note that every element of $\Omega$ belongs to $m$ sets in $\mathcal{F}$. Therefore, the number of elements of $\Omega$ not fixed by $\sigma$ is at least $(1-\alpha)|\mathcal{F}|(s-1)/m$.

Using that every set in $\mathcal{F}$ has $s$ elements and every element belongs to $m$ sets in $\mathcal{F}$, we deduce that $s|\mathcal{F}| = m|\Omega|$. Therefore, the number of elements of $\Omega$ not fixed by $\sigma$ is at least $\displaystyle{(1-\alpha)\frac{(s-1)}{s}|\Omega|}$. 
\end{proof}

\begin{corollary}\label{cor:dual-motion} Let $X$ be a non-complete geometric distance-regular graph on $n$ vertices and let $\widetilde{X}$ be its dual graph on $\widetilde{n}$ vertices. Assume that $\displaystyle{\motion(\widetilde{X})\geq \gamma \widetilde{n}}$, then $\displaystyle{\motion(X)\geq \frac{\gamma}{2}n}$.
\end{corollary}
\begin{proof} Let $\sigma$ be a non-identity element of $\Aut(X)$. Then $\sigma$ maps Delsarte cliques to Delsarte cliques. Thus $\sigma$ induces an automorphism $\widetilde{\sigma}$ of $\widetilde{X}$. Note that if $X$ is non-complete and geometric, then every vertex in $X$ is uniquely determined by the set of Delsarte cliques that contain it. Hence, if $\sigma$ is non-identity, then $\widetilde{\sigma}$ is non-identity as well. So by assumptions of the corollary, $\widetilde{\sigma}$ fixes at most $(1-\gamma)\widetilde{n}$ vertices of $\widetilde{X}$. Using that every Delsarte clique is of size at least 2, we get from the previous lemma that $\sigma$ fixes at most $(1-\gamma/2)n$ vertices.   
\end{proof}

We summarize the discussion of this section in the following theorem.

\begin{theorem}\label{thm:mu=1} Let $X$ be a geometric distance-regular graph of diameter $d\geq 2$ on $n$ vertices. Suppose $\mu =1$ and the smallest eigenvalue of $X$ is $-m$, where $m\geq 3$. Let $\varepsilon=\varepsilon(d)$ and $\eta = \eta(d)<1/2$ be constants provided by Theorem~\ref{thm:cited-main-spectral-gap}.  Assume $c_t\leq \varepsilon k$ and $b_t\leq \varepsilon k$ for some $t\in [d]$, and $k\geq \max(4m/ \eta, m^2)$. Then  $$\motion(X)\geq \frac{\eta}{4} n.$$
\end{theorem}
\begin{proof} Let $\widetilde{X}$ be a dual graph of $X$ on $\widetilde{n}$ vertices. Then, Proposition~\ref{prop:motion-dual} implies  $\motion(\widetilde{X})\geq (\eta/2)\cdot \widetilde{n}$. Therefore, the statement of the theorem follows from Corollary~\ref{cor:dual-motion}. 
\end{proof}

\subsection{Distance-regular line graphs with $\mu=1$}

Note that, by Lemma~\ref{m2-line}, geometric distance-regular graphs with smallest eigenvalue $-2$ are line graphs. Thus, we can use the following result of Mohar and Shawe-Taylor \cite{Mohar-Shave}.

\begin{definition} A distance-regular graph of diameter $d$ with parameters
\[ k = s(t+1),\  \lambda = s-1,\  c_i = 1 \text{ and } b_i = k - s \text{ (for\ $i=1, \ldots, d-1$), } c_d = t+1\]
is called a \textit{generalized $2d$-gon of order $(s, t)$}. 
\end{definition}

\begin{theorem}[Mohar, Shawe-Taylor{\cite[Theorem 3.4]{Mohar-Shave}}]\label{thm:Mohar-Shawe}
Suppose the line graph $L(Y)$ of a graph $Y$ is distance-regular. Then, either $Y$ is a Moore graph, or $Y$ is a generalized $2d$-gon of order $(1, s)$ for some $s\geq 1$, or $Y = K_{1, s}$ for $s\geq 1$.  
\end{theorem}

By the Hoffman-Singleton theorem, it is known that a Moore graph is either a complete graph, a polygon, or it is the Peterson graph ($k = 3$), the Hoffman-Singleton graph ($k=7$), or it has degree $k=57$ and diameter $d=2$.
%

Note that a generalized $2d$-gon of order $(1, s)$ has intersection numbers $a_i = 0$ for any $i\in [d]$. Thus it is bipartite. Recall, that each of two connected components of the distance-2 graph $X_2$ of a bipartite distance-regular graph $X$ is called a \textit{halved graph}. 

\begin{fact}[see {\cite[Theorem. 6.5.1]{BCN}}]\label{fact:bip-gon} If $X$ is a generalized $2d$-gon of order $(1, s)$, then $d$ is even and its halved graph is a generalized $d$-gon of order $(s, s)$. 
\end{fact} 

A celebrated theorem of W. Feit and G. Higman \cite{Feit-Higman} asserts that apart from polygons, generalized $2d$-gons exist only for $2d\in \{4, 6, 8, 12\}$. 

\begin{theorem}[W. Feit, G. Higman]\label{thm:Feit-Higman} A generalized $2d$-gon of order $(s, t)$ exists only for $2d \in\{4,6, 8, 12\}$ unless $s=t=1$. If $s>1$, then $2d\neq 12$.
\end{theorem}

Finally, we use the following bound on the zero-weight spectral radius of generalized $2d$-gon of order $(s, s)$ for $2d\leq 6$.

\begin{fact}[{\cite[Table 6.4]{BCN}}]\label{fact:gon-spectral} Let $X$ be a generalized $2d$-gon of order $(s, s)$ for $2d\leq 6$, $s>1$. Then the zero-weight spectral radius of $X$ satisfies $\xi(X)\leq 2s$.
\end{fact}

\begin{proposition}\label{prop:mu1-line} Let $X$ be a geometric distance-regular graph of diameter $d\geq 2$ on $n$ vertices. Suppose $\mu=1$, $k>4$ and the smallest eigenvalue of $X$ is $-2$. Then 
\[ \motion(X)\geq \frac{1}{16}n.\]
\end{proposition}
\begin{proof}
By Lemma~\ref{m2-line}, $X$ is a line graph. Let $\widetilde{X}$ be the dual graph of $X$. Thus, by Theorem~\ref{thm:Mohar-Shawe}, $\widetilde{X}$ is a Moore graph or a generalized $2d$-gon of order $(1, s)$ for $s=k/2>2$. 
If $\widetilde{X}$ is a Moore graph, then $\mu=1$ implies that $\widetilde{X}$ is not complete, and $k>4$ implies $\widetilde{X}$ is not a polygon. Thus $\widetilde{X}$ is a strongly regular graph. Hence, Theorem~\ref{thm:babai-srg} implies that $\motion(\widetilde{X})\geq n/8$ and the desired bound on the motion of $X$ follows from Corollary~\ref{cor:dual-motion}.

 Therefore, we may assume that $\widetilde{X}$ is a generalized $2d$-gon of order $(1, s)$ for $s>2$.  Then, by Fact~\ref{fact:bip-gon}, a halved graph $Y$ of $\widetilde{X}$ is a generalized $d$-gon of order $(s, s)$ (and $d$ is even). Moreover, by Theorem~\ref{thm:Feit-Higman}, $d\leq 6$ and by Fact~\ref{fact:gon-spectral}, $\xi(Y)\leq 2s$. Note that any pair of vertices in $Y$ has at most $q(Y) = s-1$ common neighbors. Therefore, by Lemma~\ref{mixing-lemma-tool},
\[\motion(Y)\geq \frac{s(s+1)-3s}{s(s+1)}|V(Y)|\geq \frac{s-2}{s+1}|V(Y)|\geq \frac{1}{4}|V(Y)|.\]

We note that $|V(\widetilde{X})| = 2|V(Y)|$ and $\motion(\widetilde{X})\geq \motion(Y)$ (see \cite[Prop.~5.5]{kivva-drg}). Therefore, the statement of the proposition follows from Corollary~\ref{cor:dual-motion}.  
\end{proof}
\begin{remark} We note that one can show a linear lower bound (with a worse constant) on motion in this case without using the Feit-Higman classification theorem.  Since a dual graph $\widetilde{X}$ is bipartite or of diameter 2, one can use Theorem~\ref{thm:babai-srg} and the  bounds on the motion of bipartite graphs, which we proved in \cite[Theorem 5.6 and Prop.~5.11]{kivva-drg}.  
\end{remark}

\subsection{Combining all pieces together}

Finally, we combine the results of this paper and of \cite{kivva-drg} to get the proof of Babai's conjecture (Conjecture~\ref{conj-dist-reg}). From~\cite{kivva-drg}, in addition to Theorems~\ref{thm:main-motion-cited} and~\ref{thm:cited-main-spectral-gap}, we need the following observation.

\begin{proposition}[{\cite[Proposition 4.9]{kivva-drg}}]\label{primitive-distinguish} Let $X$ be a primitive distance-regular graph of diameter $d\geq 2$. Fix some real number $\alpha>0$. Suppose that for some $1\leq j \leq d-1$ the inequalities $b_j\geq \alpha k$ and $c_{j+1}\geq \alpha k$ hold. Then  $$\motion(X)\geq \frac{\alpha}{d}n.$$
\end{proposition}

\begin{theorem} For any $d\geq 3$ there exists $\gamma_d>0$, such that for any primitive distance-regular graph $X$ of diameter $d$ on $n$ vertices either
\[ \motion(X) \geq \gamma_d n,\]
or $X$ is the Hamming graph $H(d, s)$ or the Johnson graph $J(s, d)$. 
\end{theorem}
\begin{proof} Recall, Theorem~\ref{thm:main-motion-cited} implies that either $\motion(X)\geq \gamma_d' n$ for some $\gamma_d'>0$, or $X$ is geometric with smallest eigenvalue $\geq -m_d$, for some $m_d\geq 2$. 

Let $\varepsilon' = \varepsilon(d)$ and $\eta_d = \eta(d)$ be the constants provided by Theorem~\ref{thm:cited-main-spectral-gap}. Set $$0<\varepsilon =\frac{1}{2} \min\left(\frac{1}{6m_d^4 d}, \varepsilon'\right)<1/200.$$
\begin{description} 

\item[Case A.] $X$ is not geometric or the smallest eigenvalue of $X$ is less than $-m_d$.\\ Then, by Theorem~\ref{thm:main-motion-cited}, $\motion(X)\geq \gamma_d' n$.

\item[Case B.] There exists $t\in [d]$ such that $c_{t+1}\geq \varepsilon k$ and $b_{t}\geq \varepsilon k$.\\ Then, by Proposition~\ref{primitive-distinguish}, $\motion(X)\geq \varepsilon n/d$.

\item[Case C.] $X$ is geometric with smallest eigenvalue at least $-m_d$ and there exists $t\in [d]$ such that $c_{t}\leq \varepsilon k$ and $b_t\leq \varepsilon k$.\\
By Theorem~\ref{thm:cited-main-spectral-gap}, the zero-weight spectral radius of $X$ satisfies $\xi(X)\leq (1-\eta_d)k$.

\begin{description}

\item[Case C.1.] $k< \max(29, 2m_d^3, 4m_d/\eta_d)$. \\
Then $X$ has at most $N_d = \max(29, 2m_d^3, 4m_d/\eta_d)^d+1$ vertices. Moreover, every non-trivial automorphism moves at least 2 points, so $\displaystyle{\motion(X)\geq \frac{2}{N_d}n}$.  

\item[Case C.2.] $k\geq \max(2m_d^3, 29)$ and $\mu\geq 2$.

\noindent\textbf{Case C.2.i.} $\theta_1<(1-\varepsilon)b_1$. \\ Using Corollary~\ref{cor:mu-m2} we obtain, $\displaystyle{\lambda\geq \frac{k}{m_d}-1\geq m_d^2\geq \mu}$. By Lemma~\ref{lem:lambda-mu}, we have  $2\lambda\leq \mu+k$, so  $b_1 \geq (k-\mu)/2\geq k/4$. Thus $$\xi(X)+q(X)\leq k-\varepsilon b_1\leq \left(1-\frac{\varepsilon}{4}\right)k.$$ Hence, by Lemma~\ref{mixing-lemma-tool}, $\displaystyle{\motion(X)\geq \frac{\varepsilon}{4}n}$.

\noindent\textbf{Case C.2.ii.} $\theta_1 \geq (1-\varepsilon)b_1$ and $\mu\geq 3$.\\
Since $\varepsilon<\frac{1}{200}<\varepsilon^*$, by Theorem~\ref{thm:pseudo-main-intr} and Proposition~\ref{prop:mu-eigen-bound}, $X$ is a Johnson graph. 

\noindent\textbf{Case C.2.iii.} $\theta_1 \geq (1-\varepsilon)b_1$ and $\mu=2$.\\
By Corollary~\ref{cor:Hamming}, $X$ is a Hamming graph.

\item[Case C.3.] $\mu = 1$ and $k\geq \max(4m_d/\eta_d, m_d^2)$. 

\noindent\textbf{Case C.3.i.} The smallest eigenvalue $-m$ of $X$ satisfies $-m\leq -3$.
\\Then by Theorem~\ref{thm:mu=1}, $\displaystyle{\motion(X)\geq \frac{\eta_d}{4}n}$.

\noindent\textbf{Case C.3.ii.} The smallest eigenvalue $-m$ of $X$ satisfies $-m>-3$.\\
Since, by Lemma~\ref{lem:eigenvalue-clique-geom}, $m$ is integer, we get that $m\leq 2$. Hence, $m=2$, and by Proposition~\ref{prop:mu1-line} $\motion(X)\geq n/16$.
\end{description}
\end{description}
Therefore, the statement of the theorem holds with $\displaystyle{\gamma_d = \min\left(\frac{\eta_d}{4}, \frac{\varepsilon}{d}, \frac{2}{N_d}, \gamma_d', \frac{1}{16}\right)}$.
\end{proof}

Finally, our main result on the motion, Theorem~\ref{thm:main-motion}, follows from the theorem above and the following result we proved in~\cite{kivva-drg}.  

\begin{theorem}[{\cite[Theorem 1.7]{kivva-drg}}] Assume Conjecture~\ref{conj-dist-reg} is true. Then for any $d\geq 3$ there exists $\widetilde{\gamma}_d>0$, such that for any distance-regular graph $X$ of diameter $d$ on $n$ vertices either
$$\motion(X)\geq \widetilde{\gamma}_d n,$$
or $X$ is a Johnson graph $J(s, d)$, or a Hamming graph $H(d, s)$, or a cocktail-party graph. 
\end{theorem}

\section{Outlook: Minimal degree of permutation groups. Coherent configurations}\label{sec:outlook}

The study of the minimal degree of permutation groups started in the 19th century and experienced a revival in the post-classification\footnote{We refer here to the Classification of the Finite Simple Groups (CFSG)} era.

\paragraph{Elementary results.} It seems that the minimal degree was first studied in the 1871 paper~\cite{Jordan} by Jordan, who proved that there are only finitely many primitive groups with any given minimal degree $m> 3$. A remarkable result of Bochert~\cite{Bochert} (1894) asserts that a doubly transitive permutation group of degree $n$, other than $A_n$ and $S_n$, has minimal degree at least $n/4-1$.


Lower bounds on the minimal degree impose strong structural constraints on the group. For a group $G$, following Babai \cite{babai-srg},  by the \textit{thickness} $\theta(G)$ we mean the maximal $t$ for which the alternating group $A_t$ is involved in $G$ as a quotient of a subgroup.  A classical result of Wielandt \cite{Wielandt} shows that a linear lower bound on the minimal degree of a permutation group implies a logarithmic upper bound on the thickness of the group.

\begin{theorem}[Wielandt \cite{Wielandt}, see {\cite[Theorem 6.1]{Babai-doublytransitive}}]\label{Wielandt-thm}
Let $n>k>\ell$ be positive integers, $k\geq 7$, and let $0<\alpha<1$. Suppose that $G$ is a permutation group of degree $n$ and minimal degree at least $\alpha n$. If
\[\ell(\ell-1)(\ell-2)\geq (1-\alpha)k(k-1)(k-2),\]
and $\theta(G)\geq k$, then $n\geq \binom{k}{\ell}$.
\end{theorem} 

\begin{corollary}\label{Wielandt-cor}
Let $G$ be a permutation group of degree $n$. Suppose $\mindeg(G)\geq \alpha n$. Then the thickness $\theta(G)$ of $G$ satisfies $\displaystyle{\theta(G)\leq \frac{3}{\alpha}\ln(n)}$.
\end{corollary}

It is known, as a corollary of CFSG, that doubly transitive groups $G\leq S_n$, other than $A_n$ and $S_n$, have order at most $n^{1+\log(n)}$ \cite{2-transitiveCFSG}. An elementary proof of a slightly weaker result was obtained by Babai and Pyber \cite{Babai-doublytransitive}, \cite{Pyber} using Corollary~\ref{Wielandt-cor} and Bochert's bound on minimal degree. 

In light of Wielandt's result, an immediate corollary of our Theorem~\ref{thm:main-motion} is the following.

\begin{theorem} For any $d\geq 3$ there exists an $\alpha_d>0$ such that
for any distance-regular graph $X$ of diameter $d\geq 3$ with $n$ vertices either $$\theta(\Aut(X))\leq \alpha_d \log n,$$ or $X$ is a Johnson graph, or a Hamming graph, or 
a cocktail-party graph.
\end{theorem}

\paragraph{Post-classification results.} For a permutation group $G\leq \Sym(\Omega)$ define $G^{(k)}$ to be the permutation group acting on the set of $k$-tuples $\binom{\Omega}{k}$ via the induced action. Using the CFSG, Liebeck and Saxl \cite{Liebeck}, \cite{Liebeck-Saxl} characterized all primitive permutation groups with minimal degree less than $n/3$.  They showed that a group $G$ with $\mindeg(G)<n/3$ is of the form $(A_{m}^{(k)})^{d}\leq G\leq S_{m}^{(k)}\wr S_d$, where $S_{m}^{(k)}\wr S_d$ is represented in the product action on $\binom{m}{k}^d$ elements. These groups are called \textit{Cameron groups}, as primitive groups of degree $n$ and order at least $n^{1+\log n}$ are of this form, by a result of Cameron \cite{Cameron} based on consequences of the CFSG.

Cameron groups act on ``Cameron schemes" (see below). These structures can be seen as hybrids between Hamming and Johnson graphs.

\paragraph{Coherent configurations.} Following Schur~\cite{Schur} (1933), one can relate the following combinatorial structure with a permutation group $G\leq \Sym(\Omega)$. Let $R_1, R_2, \ldots, R_k$ be the orbits of the $G$-action on $\Omega\times \Omega$ (the \textit{orbitals} of $G$). The tuple $\mathfrak{X}(G) = (\Omega, \{R_1, R_2, \ldots, R_k\})$ is called the \textit{orbital configuration} of $G$. The orbital configurations are a special case of more general combinatorial structures, called \textit{coherent configurations}, which were also essentially defined in the paper of Schur and were reintroduced later in various contexts \cite{Bose-Shimamoto}, \cite{Weisfeiler-Leman} (see~\cite{Weisfeiler}). The term ``coherent configuration" was introduced by Higman \cite{Higman}, who revived Schur's program and developed a representation theory of coherent configurations. A graph-theoretic study of coherent configurations was initiated by Babai~\cite{Babai-annals}.

\begin{definition} Let $\Omega$ be a set and $\mathfrak{R} = \{R_1, R_2, \ldots, R_d\}$ be a partition of $\Omega\times \Omega$ ($R_{i}\neq \emptyset$). The pair $\mathfrak{X} = (\Omega, \mathfrak{R})$ is called a \textit{coherent configuration} if the following properties hold.
\begin{enumerate}
\item For every $i\in [d]$, if $(x, x)\in R_i$ for some $x\in \Omega$, then $(y, z)\notin R_i$ for any $y\neq z\in \Omega$.
\item For every $i\in [d]$ there exists $i^*$ such that, if $(x, y)\in R_i$, then $(y, x)\in R_{i^*}$.
\item For every $i, j, t\in [d]$ there exists a number $p_{i,j}^{t}$, such that for any pair $(x, y)\in R_t$ there exist exactly $p_{i,j}^{t}$ vertices $z\in \Omega$ such that $(x, z) \in R_i$ and $(z, y) \in R_j$.    
\end{enumerate}
The number $d$ of classes in the partition $\mathfrak{R}$ is called the \textit{rank} of $\mathfrak{X}$.
\end{definition}

 The \textit{Cameron schemes} are the orbital configurations of the Cameron groups. 
  Babai conjectured that the following combinatorial analog of the Liebeck-Saxl classification should be true.

\begin{conjecture}[Babai]\label{conj:babai-coherent-motion} There exists an $\alpha>0$ such that for any primitive coherent configuration $\mathfrak{X}$ on the set of $n$ vertices either
\[\motion(\mathfrak{X})\geq \alpha n,\]
or $\mathfrak{X}$ is a Cameron scheme. 
\end{conjecture}

In the case of rank 3, this conjecture  follows from Babai's motion bound for strongly regular graphs \cite{babai-srg} (see Theorem~\ref{thm:babai-srg}) and  strongly regular tournaments \cite{Babai-annals}. We confirm this conjecture in the case of rank 4 in \cite{kivva-cc}.

Distance-regular graphs give rise to a special class of coherent configurations, in which relations are induced by the distance metric in the graph. Specializing Conjecture~\ref{conj:babai-coherent-motion} to the class of primitive distance-regular graphs we obtain the following statement. 

\begin{conjecture}[Babai]\label{conj:drg-uniform}
There exists an $\alpha>0$ such that for any primitive distance-regular graph $X$ of diameter $d\geq 3$ on $n$ vertices, either
\[\motion(X)\geq \alpha n,\]
or $X$ is a Hamming graph, or a Johnson graph.
\end{conjecture}

This conjecture is a stronger version of Conjecture~\ref{conj-dist-reg} and remains open.

\renewcommand{\thesection}{A}
\section{Appendix}\label{appendix}
Below we prove the inequality used in the proof of Proposition~\ref{prop:multiplicity}.
\begin{lemma} Let $2\leq t\leq x+1\leq m$ be integers, then
\begin{equation}\label{eq:appendix-ineq}
\frac{(m-x)(m-1)^2}{(m-x+t-2)^2}\geq \frac{m-1}{t-1}.
\end{equation} 
\end{lemma}
\begin{proof}
Note that when $x = m-1$ the inequality is true, as $m-1\geq t-1$. We can rewrite inequality~\eqref{eq:appendix-ineq} as
\[ (m-x)(m-1)(t-1)\geq (m-x+t-2)^2,\]
\[ m(m-1)(t-1)-x(m-1)(t-1)\geq x^2-2x(m+t-2)+(m+t-2)^2,\]
\begin{equation}\label{eq:app-1}
 m(m-1)(t-1)-(m+t-2)^2\geq x(x+m(t-3)-3t+5)
\end{equation}
If $t\geq 4$, then $x\geq 3t-5-m(t-3)$. Indeed, for $t\geq 5$ this is true as 
\[ x\geq t-1\geq 3t-5-2m\geq 3t-5-m(t-3),\]
and for $t=4$ this holds as $x\geq  t-1 = 3 \geq 7-m$. Thus, for $t\geq 4$ the maximal value of the RHS of inequality~\eqref{eq:app-1} is achieved at maximal value of $x$, i.e., when $x=m-1$. But as noted above, inequality~\eqref{eq:appendix-ineq} holds for $x = m-1$,  and inequality~\eqref{eq:app-1} is equivalent to it. 

The statement of the lemma is obvious if $t=2$. Therefore, the only case we still need to check is $t=3$. Since the desired inequality holds for $x=m-1$, we can assume $x\leq m-2$. In this case inequality~\eqref{eq:app-1} follows from
\[ 2m(m-1)-(m+1)^2 = m^2-4m+1 > (m-2)^2-4 \geq x^2-4\geq x^2-4x.\]
\end{proof}

\end{document}